\documentclass[11pt]{amsart}
\usepackage[margin=1.0in]{geometry}
\usepackage{amsmath,amssymb,amsfonts,graphicx,color,fancyhdr,psfrag,comment,enumerate}
\usepackage[latin1]{inputenc}
\usepackage[hyperpageref]{backref}
\usepackage[colorlinks=true, pdfstartview=FitV, linkcolor=blue, citecolor=blue, urlcolor=blue]{hyperref}
\allowdisplaybreaks

\usepackage{xcolor}
\usepackage{exscale}
\usepackage{latexsym}
\usepackage{mathtools}
\usepackage{srcltx}
\usepackage{esint}

\usepackage{epstopdf}
\usepackage{mathrsfs}
\usepackage{epsfig}
\usepackage{hyperref}
\usepackage{ifthen}
\usepackage{float}
\usepackage{enumerate}
\usepackage{mathtools}
\usepackage[bottom]{footmisc}
\usepackage{tikz}
\usetikzlibrary{matrix,shapes,arrows,positioning,chains}
\usepackage{amssymb,amsmath,amsfonts}
\usepackage{bbm}

\numberwithin{equation}{section}
\newcommand{\bigzero}{\mbox{\normalfont\Large 0}}

\DeclareMathOperator\ve{\varepsilon}
\DeclarePairedDelimiter\ev{\big\langle}{\big\rangle}%
\DeclarePairedDelimiter\ek{\Big\vert}{\Big\vert}%
\DeclarePairedDelimiterX{\pro}[2]{(}{)}{#2^{-1}\circ #1}

\newcommand{\HLM}{Hardy-Littlewood maximal function}
\newcommand{\KFP}{Kolmogorov--Fokker--Planck }
\newcommand{\bR}{\mathbb{R}^n}
\newcommand{\RR}{\mathbb{R}}
\newcommand{\NN}{\mathbb{N}}
\newcommand{\PP}{\mathcal{P}}
\newcommand{\GG}{\mathscr{G}}

\newcommand{\ZZ}{\mathbb{Z}}
\newcommand{\TT}{\mathcal{T}}
\newcommand{\SSS}{\mathcal{S}}
\newcommand{\UU}{\mathscr{U}}
\newcommand{\Om}{\Omega}
\newcommand{\BB}{\mathcal{B}}

\newcommand{\ptl}{\partial}

\def\fr{\frac}
\newcommand{\CZ}{Calder\'on-Zygmund}

\DeclareMathOperator*{\esssup}{ess\,sup}

\newcommand{\descitem}[2]{\item[{(#1)}]\label{#2}}
\newcommand{\descref}[2]{\hyperref[#1]{\textnormal{\textcolor{black}{(}\textcolor{blue}{\bf #2}\textcolor{black}{)}}}}

\usepackage[capitalize,nameinlink]{cleveref}
\crefname{section}{Section}{Sections}
\crefname{subsection}{Subsection}{Subsections}
\crefname{condition}{Condition}{Conditions}
\crefname{hypothesis}{Hypothesis}{Hypothesis}
\crefname{assumption}{Assumption}{Assumptions}
\crefname{lemma}{Lemma}{Lemmas}
\crefname{claim}{Claim}{Claims}
\crefname{remark}{Remark}{Remarks}

\crefformat{equation}{\textup{#2(#1)#3}}
\crefrangeformat{equation}{\textup{#3(#1)#4--#5(#2)#6}}
\crefmultiformat{equation}{\textup{#2(#1)#3}}{ and \textup{#2(#1)#3}}
{, \textup{#2(#1)#3}}{, and \textup{#2(#1)#3}}
\crefrangemultiformat{equation}{\textup{#3(#1)#4--#5(#2)#6}}%
{ and \textup{#3(#1)#4--#5(#2)#6}}{, \textup{#3(#1)#4--#5(#2)#6}}%
{, and \textup{#3(#1)#4--#5(#2)#6}}

\Crefformat{equation}{#2Equation~\textup{(#1)}#3}
\Crefrangeformat{equation}{Equations~\textup{#3(#1)#4--#5(#2)#6}}
\Crefmultiformat{equation}{Equations~\textup{#2(#1)#3}}{ and \textup{#2(#1)#3}}
{, \textup{#2(#1)#3}}{, and \textup{#2(#1)#3}}
\Crefrangemultiformat{equation}{Equations~\textup{#3(#1)#4--#5(#2)#6}}%
{ and \textup{#3(#1)#4--#5(#2)#6}}{, \textup{#3(#1)#4--#5(#2)#6}}%
{, and \textup{#3(#1)#4--#5(#2)#6}}

\crefdefaultlabelformat{#2\textup{#1}#3}
\newtheorem{theorem}{Theorem}[section]
\newtheorem{lemma}[theorem]{Lemma}

\newtheorem{proposition}[theorem]{Proposition}

\newtheorem{definition}[theorem]{Definition}
\newtheorem{remark}[theorem]{Remark}        

\numberwithin{equation}{section}

\newcommand{\redref}[2]{\texorpdfstring{\protect\hyperlink{#1}{\textcolor{black}{(}\textcolor{red}{#2}\textcolor{black}{)}}}{}}
\newcommand{\redlabel}[2]{\hypertarget{#1}{\textcolor{black}{(}\textcolor{red}{#2}\textcolor{black}{)}}}

\AtBeginDocument{%
\definecolor{armygreen}{rgb}{0.15, 0.75, 0.0}
\hypersetup{citecolor=armygreen}}

\begin{document}
\title[Pointwise Estimates for Kolmogorov--Fokker--Planck operators]{Pointwise and Weighted Hessian Estimates for Kolmogorov--Fokker--Planck type operators}
\author[A. Ghosh]{Abhishek Ghosh}
\email{abhi21@tifrbng.res.in}
\author[V. Tewary]{Vivek Tewary}
\email{vivek2020@tifrbng.res.in}
\address{Tata Institute of Fundamental Research, Centre for Applicable Mathematics, Bangalore--560065, Karnataka, India.}

\begin{abstract}
In this article, we obtain hessian estimates for \KFP operators in non-divergence form in several Banach function spaces. Our approach relies on a representation formula and newly developed sparse domination techniques in Harmonic Analysis. Our result when restricted to weighted Lebesgue spaces yields sharp quantitative hessian estimates for the \KFP operators.
\end{abstract}
{\let\thefootnote\relax}

\thanks{Authors are supported by Centre for Applicable Mathematics, Tata Institute of Fundamental Research.}
\maketitle
\section{Introduction}

In this article, we consider a class of Kolmogorov--Fokker--Planck type operator on
$\mathbb{R}^{N+1}$.
\begin{align}\label{maineq}
\displaystyle  \mathcal{L}u \equiv
\sum_{i,j=1}^{s_0}{\,a_{ij}(x, t)\ptl_{x_i}\ptl_{x_j}u
}+\sum_{i,j=1}^N b_{ij}x_i {\ptl_{x_j} u }- {\ptl_t
\,u}=0,
\end{align}
where $z=(x,t)\in { \RR}^{N+1}$, $1\leq s_0 \leq N$, and $b_{ij}$ is
constant for every $i,j= 1,\cdots, N$. The following
assumptions on the coefficients of $\mathcal{L}$:
\begin{description}
\descitem{H1}{H1}  $a_{ij}=a_{ji} \in L^{\infty} ({\RR}^{N+1})$ and there
exists a $\lambda >0$ such that
$$
\fr{1}{\lambda}\sum_{i=1}^{s_0}\xi_i^2 \leq \sum_{i,j=1}^{s_0}
a_{ij}(x, t)\xi_i \xi_j \leq {\lambda}\sum_{i=1}^{s_0}\xi_i^2
$$
for every $(x, t)\in {\RR}^{N+1}$, and $\xi \in {\RR}^{s_0}$.

\descitem{H2}{H2} The matrix $B=(b_{ij})_{N \times N}$ has the form
$$\left(
\begin{array}{cccccc}
0 & {B_1} & 0 & \cdots & 0 \\
0 & 0  & {B_2} & \cdots & 0 \\
\vdots & \vdots & \vdots & \ddots &\vdots \\
0 & 0 & 0 & \cdots & {B_d} \\
0 & 0 & 0 & \cdots & 0
\end{array}
\right)
$$
where $B_k$ is a matrix $s_{k-1}\times s_{k}$ with rank $s_k$ and
$s_0\geq s_1\geq \cdots\geq s_d$, $s_0+s_1+\cdots+s_d=N$.
\end{description}
It is well known that under the assumptions \descref{H1}{H1} and \descref{H2}{H2}, for any fixed $z_{0}\in \RR^{N+1}$ the ``frozen" operator 
\begin{align}\label{maineq2}\displaystyle  \mathcal{L}_0 \equiv
\sum_{i,j=1}^{s_0}{\,a_{ij}(z_0)\ptl_{x_i}\ptl_{x_j}u
}+\sum_{i,j=1}^N b_{ij}x_i {\ptl_{x_j} u }- {\ptl_t
\,u}\end{align}
 satisfies the well-known
H\"ormander's hypoellipticity condition. A fundamental property of \cref{maineq2} is that the fundamental solution can be explicitly written down since the
coefficients $a_{ij}$ are constant. In fact, it is well known that hypoelliptic operators with smooth coefficients satisfying H\"ormander's hypoellipticity condition are locally solvable \cite{Bram14}. This property has been useful, first in the Euclidean setting, to obtain the Hessian estimates on $L^p$-spaces for elliptic and parabolic equations in non-divergence form by means of Harmonic Analysis techniques, particularly, the analysis of \CZ~operators and commutators in the papers \cite{Chiarenza1991, BramCer1993} and then in the setting of \KFP operators \cite{Bramanti-JMAA-ultraparabolic}.

This present article addresses hessian estimates for Kolmogorov--Fokker--Planck type operators on several function spaces, for example, weighted $L^p$ spaces, generalized Orlicz spaces and variable Lebesgue spaces. To the best of our knowledge, hessian estimates for \cref{maineq} are not known beyond $L^p$ spaces.  The regularity theory of \KFP equations has become an area of extensive research recently, particularly due to its connections with kinetic equations that exhibit a variety of interesting physical and mathematical features \cite{silvestreRegularityEstimatesOpen2022}. In the ultraparabolic setting, one needs to work with Sobolev spaces adapted to the appropriate vector fields associated to \KFP operators. Indeed, the hypoelliptic structure of the equation determines the directions in which the equation has regularity properties. Our point of departure are the papers \cite{Bramanti-JMAA-ultraparabolic} and \cite{Poli1998} where unweighted versions of these results were proved. Other related works are \cite{Bramanti2013} in the non-divergence setting and \cite{ManPoli1998} for the divergence type equations. Two recent works that study weighted versions of $L^p$ estimates for the kinetic \KFP equations are \cite{Dong2021} and \cite{Niebel2022}. 

An important thrust of this work is to employ new techniques in Harmonic Analysis in the context of \KFP~operators to obtain pointwise estimates of the hessian. In the Euclidean setting, such techniques have been previously used in \cite{Bui-CCM,Bui-IMRN-parabolic}. Our work extends their methods in the following particulars:
\begin{itemize}
    \item The setting of \KFP operators requires techniques specific to homogeneous spaces. In particular, much of our analysis would rely on the existence of Christ's dyadic grids in metric spaces with doubling measures.
    \item Our focus is on obtaining sharp weighted estimates in Lebesgue and variable exponent Lebesgue spaces. See the discussion in \cref{sharpes}.
    \item Our stress is also on the unified nature of the estimates. Hence, our pointwise estimates can give sharp results in more general Banach function spaces. See \cref{Banach}.
\end{itemize}
Now we present our results systematically. Our first result addresses sharp quantitative weighted estimates on weighted Lebesgue spaces for \KFP~operators. 

\subsection{Sharp regularity estimates}\label{sharpes}
Sharp weighted estimates for the gradient are always of immense interest in the context of partial differential equations, for example we refer the work of Astala--Iwaniec--Saksman \cite{Astala-Iwaniec-Eero} where the authors have studied the connections between the quasi regularity of solutions of Beltrami equation with the sharp operator norm estimate of the Ahlfors--Beurling transform. Weighted estimates for the non-divergence elliptic equations were first addressed in \cite{Non-div-elliptic-Byun-2015} where they have proved $\|D^2 u\|_{L^p(\omega)}\lesssim \|f\|_{L^p(\omega)}$ for $2<p<\infty$ and $\omega\in A_{p/2}(\bR)$. Very recently, the estimate $\|D^2 u\|_{L^p(\omega)}\lesssim \|f\|_{L^p(\omega)}$ is improved for all $1<p<\infty$ and all $\omega\in A_{p}(\bR)$ by the authors in \cite{Bui-CCM}. They have employed the sparse domination technique from Harmonic analysis. Subsequently, the same techniques are employed to the parabolic setting in \cite{Bui-IMRN-parabolic} to obtain regularity estimates on generalized Orlicz space. Previously, regularity estimates on generalized Orlicz spaces for elliptic equations of non-divergence type was obtained by H\"ast\"o and Ok in \cite{Hasto-JDE-elliptic}. In this direction, we obtain the following sharp weighted regularity estimates on weighted $L^p$ spaces.
\begin{theorem}
\label{sharp-weighted-estimates}
Assume that the matrices $A$ and $B$ satisfy \descref{H1}{H1} and \descref{H2}{H2}. There exists $\delta$ such that if $A$ is $(\delta,\mathfrak{K})$-$\text{BMO}$ in the sense of \cref{BMOA} then for any $\Omega'\Subset \Omega\subset\RR^{N+1}$, $1<p<\infty$, and $\omega\in \mathcal{A}_{p} (\cref{Muckenhoupt-ultraparabolic})$ we have the following estimate
\begin{align}
\label{sharp-gradient-estimates}
    \sum_{i,j=1}^{s_0}\|u_{x_ix_j}\|_{L^p(\Omega',\omega)}\leq \kappa_{A, C, N, \delta}\,  [\omega]_{\mathcal A_{p}}^{\max\left\{\frac{1}{p-1}, 1\right\}}\left(\|\mathcal{L} u \|_{L^p(\Omega,\omega)}+\|u \|_{L^p(\Omega,\omega)}\right).
\end{align} 
for any $u\in \mathcal{S}^p_{\omega}(\Omega,\mathcal{L})$ (\cref{sobolevL}) such that $u$ is supported in $\Omega\cap\{z=(x,t)\in\RR^{N+1}:t>0\}$.
\end{theorem}
We would like to highlight certain key features of Theorem~\ref{sharp-weighted-estimates}. 
\begin{itemize}
\item
The exponent $\max\left\{\frac{1}{p-1}, 1\right\}$ in \eqref{sharp-gradient-estimates} is sharp. Let us consider the following equation
$$\Delta u=f\,\, \text{on}\,\, \mathbb{R}^n.$$
By means of Fourier transform we can write the hessian as $\nabla^2 u=(\mathcal{R}\otimes \mathcal{R}) (f)$, where $\mathcal{R}=(R_1, \cdots, R_n)$ is the Riesz transform. Therefore, using estimates from \cite{Benea-Bernicot}, we obtain $\|\nabla^{2}u\|_{L^p(\omega)}\lesssim [\omega]_{A_p(\mathbb{R}^n)}^{\max\{1, \frac{1}{p-1}\}}\|f\|_{L^p(\omega)}$, which validates the sharpness of our result \eqref{sharp-gradient-estimates}. 
\item
Our techniques of the proof of Theorem~\ref{sharp-weighted-estimates} are also applicable for elliptic equations of non-divergence form or parabolic equations of non-divergence form. Therefore, the estimate \eqref{sharp-gradient-estimates} is new even for elliptic equations of non-divergence form or parabolic equations of non-divergence form.

Previous results in this directions are motivated by the seminal works \cite{Caff-Ann-1990, Caff-Peral}. Most of these techniques heavily rely on the Fefferman-Stein sharp maximal function, small perturbation arguments, good-$\lambda$-type inequalities and reverse H\"older's inequality( self-improving property) of Muckenhoupt weights. Therefore, it is really not clear how to obtain sharp estimates using these methods. On the other hand, our proof relies on the representation formula \eqref{represent3} developed in \cite{Bramanti-JMAA-ultraparabolic}. We improve the representation formula \eqref{represent3} to \eqref{represent3-sparse}, where $u_{x_{i} x_{j}}$ is pointwise dominated by appropriate sparse operators (see Definition~\ref{defition-sparse-operator}). The primary usefulness of the sparse domination lies in the fact that it is easy to get quantitative weighted estimates for sparse operators which can be immediately passed on to the operator at hand. It should be noted that the exponent in \eqref{sharp-gradient-estimates} is same as that of the celebrated $A_2$ conjecture (now a Theorem \cite{Hyt-Ann}).
\end{itemize}
As mentioned above, our techniques also imply hessian estimates on several Banach function spaces. We state them more precisely.
\subsection{Estimates on Banach function spaces}\label{Banach}
The study of regularity estimates for non-divergence type elliptic and parabolic equations on generalized Orlicz spaces and variable Lebesgue spaces has gained a lot of attention over past decades, due to its physical as well as mathematical importance( see \cite{Mingione-ARMA-2001}, \cite{Byun-TAMS-variable}, \cite{Byun-Math-Annalen-Variable}, \cite{Hasto-JDE-elliptic}, \cite{Byun-Wang}, \cite{Zhikov86}) we believe our results, that is, \cref{main-theorem} and \cref{main-theorem2} will be of serious interest.
\begin{theorem}
\label{main-theorem}
Suppose that the matrices $A$ and $B$ satisfy \descref{H1}{H1} and \descref{H2}{H2}. There exists $\delta$ such that if $A$ is $(\delta,\mathfrak{K})$-$\text{BMO}$ in the sense of \cref{BMOA} then for any $\Omega'\Subset \Omega\subset\RR^{N+1}$, there exists a positive constant $\kappa$ depending on $p, q, \alpha, \beta, \lambda, \Omega', \Omega$, the matrix $B$ and $\mathfrak{K}$ such that for any $u\in S^{\varphi(\cdot)}(\Omega,d,|\cdot|)$ with support in $\Omega\cap\{z=(x,t)\in\RR^{N+1}:t>0\}$, we have
\begin{align}
    \|u\|_{S^{\varphi(\cdot)}(\Omega',d,|\cdot|)}\leq C\left(\|\mathcal{L}u\|_{L^{\varphi(\cdot)}(\Om,d,|\cdot|)}+\|u\|_{L^{\varphi(\cdot)}(\Om,d,|\cdot|)}\right)
\end{align} where $L^{\varphi(\cdot)}(\Om,d,|\cdot|)$ is the generalized Orlicz space satisfying conditions \descref{A0}{A0}, \descref{A1}{A1}, \descref{aInc$_p$}{aInc$_p$}, and \descref{aDec$_q$}{aDec$_q$}. The definition of the Sobolev space $S^{\varphi(\cdot)}(\Omega,d,|\cdot|)$ adapted to the ultraparabolic operator $\mathcal{L}$ may be found in \cref{sobolevL}. The definitions regarding generalized Orlicz spaces may be found in \cref{SparseDomGenOrc}. 
\end{theorem}

\begin{theorem}
\label{main-theorem2}
Suppose that the matrices $A$ and $B$ satisfy \descref{H1}{H1} and \descref{H2}{H2}. There exists $\delta$ such that if $A$ is $(\delta,\mathfrak{K})$-$\text{BMO}$ in the sense of \cref{BMOA} then for any $\Omega'\Subset \Omega\subset\RR^{N+1}$, there exists a positive constant $\kappa$ depending on $p(\cdot), \lambda, \Omega', \Omega$, the matrix $B$ and $\mathfrak{K}$ such that for any $u\in S^{p(\cdot)}(\Omega,\omega)$ with support in $\Omega\cap\{z=(x,t)\in\RR^{N+1}:t>0\}$, we have
\begin{align}
    \|u\|_{S^{p(\cdot)}(\Omega',\omega)}\leq \kappa_{A, B, N, \delta,\gamma_1,\gamma_2,\omega}\left(\|\mathcal{L}u\|_{L^{p(\cdot)}(\Om,\omega)}+\|u\|_{L^{p(\cdot)}(\Om,\omega)}\right)
\end{align} where $L^{p(\cdot)}(\Om,\omega)$ is the variable exponent Lebesgue space such that the variable exponent $p(\cdot)$ is globally log-H\"older continuous and $\omega\in \mathcal{A}_{p(\cdot)}$. The definition of the Sobolev space $S^{p(\cdot)}(\Omega,\omega)$ adapted to the ultraparabolic operator $\mathcal{L}$ may be found in \cref{sobolevL}. The definitions regarding weighted variable exponent Lebesgue spaces may be found in \cref{sparsedomVarLeb}. 
\end{theorem}

We write down the particulars of our approach below:
\begin{enumerate}
    \item Using the fundamental solution of the frozen operator \cref{maineq2}, we obtain a pointwise representation of the hessian ($u_{x_i x_j}$) for any test function \cref{represent} in terms of (i) a singular operator acting on $\mathcal{L}u$, and, (ii) a commutator of the same singular operator with the coefficients $a_{ij}$ acting on the hessian of $u$.
    \item In a second step, the variable kernel singular operator is expanded in spherical harmonics (\cite{CaldZyg57}, \cite{Bramanti-JMAA-ultraparabolic}) to obtain a second representation formula \cref{represent3}.
    \item At this point, our approach departs from that of \cite{Bramanti-JMAA-ultraparabolic}. We first prove endpoint boundedness of the singular operators (weak-type $L^1$) and more importantly, the appropriate end-point boundedness of the associated grand truncated maximal operator.
    \item We derive pointwise estimates for the singular integrals and the commutators in terms of some positive dyadic operators, known as ``sparse operators".
    \item Combining the above, we obtain a further representation formula \cref{represent3-sparse} for the hessian of $u$ in terms of sparse operators which forms our basis for a unified approach to gradient estimates. As mentioned in the work \cite{Bui-IMRN-parabolic}, this approach indeed improves all the known estimates and have further implications in the regularity theory of PDEs.
\end{enumerate}

\medskip


\subsection{History of the problem}

\KFP equations arise in a variety of physical models in mathematical finance \cite{Barles1997,Anceschi2021} among other stochastic models, kinetic theory of gases \cite{PLLions1994,Alexandre2004}, clusters in space \cite{Chandra1943}, and image processing \cite{Mumford1994,AugZuck2003}. A variety of other applications are given in the books \cite{Risken1989,Honerkamp1994,GrasvanHer1999}.

On the mathematical side of things, Kolmogorov \cite{Kolmo1934} studied a simpler equation as a model for Brownian motion and wrote down its fundamental solution. Generalized models were considered in \cite{Weber1951,Ilin1964}. These studies motivated H\"ormander to build the notion of hypoellipticity in \cite{hormander67} where he finds a sufficient condition for an operator $\mathcal{L}$ to be hypoelliptic in terms of a rank condition on the Lie Algebra generated by the distinct first order differential operators in $\mathcal{L}$ now interpreted as vector fields. 
Subsequently, Stein laid out a program of studying hypoelliptic operators by means of left invariant homogeneous operators on  nilpotent graded
Lie groups \cite{Stein1971} which resulted in such deep work as \cite{RothStein1976}, \cite{Folland1975}, and \cite{NaStWa1985}. Such analysis has found applications in spectral theory of Fokker-Planck type operators with connections to semiclassical analysis \cite{HelNi2005}.

\subsection{Regularity theory of \KFP equations}
The regularity theory for equations in nondivergence form is well-developed in the case when the coefficients are H\"older continuous. In the particular case of \KFP operators, fundamental solution is constructed by Levi's method of parametrices in \cite{Poli1994}. A solution to Dirichlet problem in the Perron-Wiener sense as well as Schauder estimates are obtained in \cite{Manfre1997}. Cauchy problem is studied in \cite{Tersenov2005}. Schauder estimates, Harnack inequality and Gaussian lower bounds are studied in \cite{DiFranPoli2006}. For less regular coefficients, the regularity theory is much less well developed. For example, the Krylov-Safanov theory for measurable coefficients is missing. Recently, Harnack's inequality has been proved for less regular coefficients in \cite{Abedin2019}.

The \CZ\,theory has had an important contribution to the  $W^{2,p}$-existence theory for equations in nondivergence form. An early work consisting of an existence result involving discontinuous coefficients was the work of Miranda \cite{Miranda1963} where solutions in $W^{2,2}\cap W^{1,2}_0$ were found for coefficients in $W^{1,N}$. This work can be strengthened by asking for derivatives of coefficients to belong to weak $L^N$. With the advent of viscosity solutions, Caffarelli extended the notion to $C^{1,\alpha}, C^{2,\alpha}$ and $W^{2,p}$-versions in \cite{Caff1989}. However, these works were restricted to $p>N-\ve$ due to the application of the Aleksandrov-Bakel'man-Pucci maximum principle. In the important work of Chiarenza et al \cite{Chiarenza1991,Chiarenza1993}, $W^{2,p}$-existence results were proved for the full range $1<p<\infty$ through a representation formula. Indeed these methods were generalized to parabolic equations in \cite{BramCer1993} and for ultraparabolic operators in \cite{Bramanti-JMAA-ultraparabolic}, from where we borrow the representation formula. Krylov \cite{Krylov2007} removed the dependence on the representation formula by exploiting the sharp maximal function of the second derivatives which allows for generalizations to fully nonlinear settings. Other than these approaches, there is also, what is called as, the maximum function free approach of Acerbi and Mingione \cite{AceMin2007} which has been developed for equations in non-divergence form in \cite{Yao2012}. 

The results concerning hypoelliptic operators are local in nature, for example, see \cite{Bramanti-JMAA-ultraparabolic,BramZhu2013}. An analysis at the boundary would require the development of the notion of {\it characteristic points} in this context since boundary regularity is expected to fail at such points as noticed by Jerison \cite{Jerison1981I,Jerison1981II} for other models.

 On the other hand, regularity theory for divergence-form variants of \KFP equations is much well-developed. We point to the survey \cite{PolidoroSurvey2020} for references to these works. For boundary behaviour of non-negative solutions of \KFP we refer the foundational works \cite{nystrom-2010, nystrom-2012, nystrom-2016}. A recent contribution to a nonlinear \KFP model is \cite{PrashantaNyst2022}.

\section{Preliminaries and Notations}
\subsection{Geometry of the problem}
We recall the following translation and dilation structures in the space $\RR^{N+1}$ intrinsic to the constant matrix $B$.

\begin{definition}[The group of translations]\label{translation}
Let $F(\tau)=exp(-\tau B^T)$. Since the matrix $B$ is a nilpotent matrix,  $F(\tau)$ is a polynomial of degree $B$ in $\tau$. For $(x, t), (y, \tau) \in \RR^{N+1}$, set $$(x, t)\circ
(y, \tau)=(y+F(\tau)x, t+\tau).$$ Then $(\RR^{N+1}, \circ)$ is a group
with identity element $(0,0)$; the inverse of an element $(x, t)$ is
$(x, t)^{-1}=(-F(-t)x, -t)$. The left translation by $z'$
given by
$$z\mapsto z'\circ z,$$
is a invariant translation to operator $\mathcal{L}_0$.
\end{definition}

\begin{definition}[Dilations]
The compatible dilation to operator $\mathcal{L}_0$ is given by
$$
\delta_{r}=diag(r I_{s_0}, r^3
I_{s_1},\cdots, r^{2d+1}I_{s_d}, r^{2}),
$$
where $I_{s_k}$ denotes the $s_k\times s_k$ identity matrix. With respect to this dilation the operator $\mathcal{L}_0$ satisfies the following
\begin{align*}
\mathcal{L}_0(u(\delta_{r}))=r^{2} \delta_{r}(\mathcal{L}_0(u)),     
\end{align*}
therefore, $\mathcal{L}_0$ is a homogeneous operator of degree 2. Let
$$Q=s_0+3s_1+\cdots+(2d+1)s_d.$$
The quantity $Q+2$ is known as the homogeneous dimension of $\RR^{N+1}$ with respect to the dilation $\delta_{r}$.
\end{definition}

\begin{definition}[Homogeneous norm]\label{homnorm}
The norm in $\RR^{N+1}$, related to the group of translations and
dilation to the equation is defined by $$||(x, t)||=\rho$$ if $\rho$ is the
unique positive solution to the equation
$$
\fr{x_1^2}{\rho^{2 a_1}}+\fr{x_2^2}{\rho^{2 a_2}}+\cdots+\fr{x_N^2}{\rho^{2 a_N}}
+\fr{t^2}{\rho^4}=1,
$$
where $(x, t) \in \RR^{N+1}\setminus \{0\}$ and
$$
a_1=\cdots=a_{s_0}=1, \quad
a_{s_0+1}=\cdots=a_{s_0+s_1}=3,\cdots,
$$
$$
a_{s_0+\cdots+s_{d-1}+1}=\cdots=a_N=2d+1,
$$
and $||(0,0)||=0$. 
\end{definition}

The homogeneous norm defined above satisfies the following property. 

\begin{lemma}(\cite[Prop. 1.3]{Bramanti-JMAA-ultraparabolic})\label{orderbeta}There exists $\beta=\beta(d)\in (0,1]$, and constants $\Tilde{K},\, M$ depending on the matrix $B$ such that for every $z,\eta,\zeta\in \RR^{N+1}$ satisfying $\|\eta^{-1}\circ z\|\geq M\|\zeta^{-1}\circ z\|$, it holds that
\begin{align}\label{order1}\|\eta^{-1}\circ z - \eta^{-1}\circ \zeta\|\leq \tilde{K}\|\zeta^{-1}\circ z\|^\beta\|\eta^{-1}\circ z\|^{1-\beta},\end{align}
\begin{align}\label{order2}\|z^{-1}\circ \eta - \zeta^{-1}\circ \eta\|\leq \tilde{K}\|\zeta^{-1}\circ z\|^\beta\|\eta^{-1}\circ z\|^{1-\beta},\end{align}
\end{lemma}

A homogeneous space satisfying \cref{orderbeta} is called a space of order $\beta$
.
\begin{definition}[Homogeneous metric]
Let $d$ denotes the metric induced by the norm $||\cdot||$, i.e., $d(z, \xi)=||\xi^{-1}\circ z||$, then $d$ is a quasimetric. We refer Proposition~1.3 in \cite{Bramanti-JMAA-ultraparabolic} for more properties of $d$. Also, we remark that the metric $d$ is not symmetric but quasi-symmetric i.e. there exist a constant $\tilde{D}$ depending on the matrix $B$ such that $$\frac{1}{\tilde{D}}\, d(z, \xi)\leq \Tilde{d}(\xi, z)\leq \tilde{D}\ d(z, \xi).$$ We define the new metric, $\Tilde{d}(z, \xi)=d(z, \xi)+d(\xi, z)$, which is comparable to the metric $d$ and is symmetric. Therefore, here on wards we assume the metric $d$ to be symmetric. 

Therefore the space $(\RR^{N+1}, d, |\cdot|)$ is a metric space of homogeneous type in the sense of Coifman and Weiss (see \cite{Coifman-Weiss-book-1971}) .  
\end{definition}

\begin{definition} A ball centered at $z=(x,t)\in\RR^{N+1}$ is defined as \begin{align}\label{balldef}B_r(z)=\{w\in \RR^{N+1}:d(z,w)<r\}.\end{align}
\end{definition}
\begin{definition}[BMO spaces]
A function $b\in L^1_{\text{loc}}(\RR^{N+1})$ is said to be in BMO if 
\begin{align*}
    ||b||_{\text{BMO}}:= \sup\left\{\frac{1}{|\BB|}\int\limits_{\BB}|b(z)-(b)_{\BB}|\,dz: \BB\text{ is a ball in }\RR^{N+1}\right\}<\infty,
\end{align*}
\end{definition}        
where $(b)_{\BB}$ denotes the mean of $b$ over the set $\BB$, that is \[(b)_{\BB}:=\fint\limits_{\BB} b(z)\,dz=\frac{1}{|\BB|}\int\limits_{\BB} b(z)\,dz,\] and $L^1_{\text{loc}}(\RR^{N+1})$ has the usual meaning of the function space of locally integrable functions with respect to the Lebesgue measure. We also use the notation $\Omega_{\BB}(b)$ to denote the oscillation $\frac{1}{|\BB|}\int\limits_{\BB}|b(z)-(b)_{\BB}|$.
\begin{definition}\label{BMOA}
For given $\mathfrak{K}>0$ and $\delta>0$, we say that a matrix $A$ whose entries are $L^1_{\text{loc}}(\RR^{N+1})$ functions is $(\delta,\mathfrak{K})-BMO$ if 
\begin{align*}
    \sup\left\{ \fint\limits_{\BB_\rho} |A(z)-(A)_\rho|\,dz: \BB_\rho \mbox{ is a ball of radius $\rho$ in }\RR^{N+1}, 0<\rho<\mathfrak{K} \right\}<\delta,
\end{align*}
\end{definition}
where $(A)_{\rho}$ denotes the mean of $A$ over a ball $\BB_\rho$ of radius $\rho$, that is \[(A)_{\rho}:=\fint\limits_{\BB_\rho} A(z)\,dz=\frac{1}{|\BB_\rho|}\int\limits_{\BB_\rho} A(z)\,dz.\]

\begin{definition}\label{sobolevL}
For $\Omega\subset\RR^{N+1}$, an open set, we define a class of Sobolev-type space \[\SSS^X(\Omega,\mathcal{L})=\{u\in X(\Omega):u_{x_i},u_{x_ix_j},Yu\in X(\Omega),i,j=1,2,\ldots,s_0\},\] where $Yu=\left(\sum_{i,j=1}^N b_{ij}x_i\partial_{x_j}-\partial_t\right)u$ and $X(\Omega)$ is one of the following Banach function spaces: $X(\Omega)={L^{p(\cdot)}_w(\Omega)}$ or ${L^{\varphi(\cdot)}(\Omega)}$, whose definitions are given in \cref{sparsedomVarLeb} and \cref{SparseDomGenOrc} respectively. On the space $\SSS^X(\Omega,\mathcal{L})$, we define the norm 
\[||u||_{\SSS^X(\Omega,\mathcal{L})}=||u||_{X(\Omega)}+\sum_{i=1}^{s_0}||u_{x_i}||_{X(\Omega)}+\sum_{i,j=1}^{s_0} ||u_{x_ix_j}||_{X(\Omega)}+||Yu||^p_{X(\Omega)}.\] 
We may also sometimes use the notation $S^X(\Omega,d,|\cdot|)$ or $S^X(\Om)$. In the specific case of $X=L^{p(\cdot)}_{\omega}(\Omega)$, we denote $\SSS^X(\Omega,\mathcal{L})$ by $\SSS^{p(\cdot)}_{\omega}(\Omega)$ and when $X=L^{\varphi(\cdot)}(\Omega)$, we denote $\SSS^X(\Omega,\mathcal{L})$ by $\SSS^{\varphi(\cdot)}(\Omega)$.
\end{definition}

A complete exposition of Cald\'eron-Zygmund theory on homogeneous spaces is given in the book of Deng and Han \cite{DengHan2009}. For our purposes, we shall require the Cald\'eron-Zygmund type decomposition adapted to $(\RR^{N+1}, d, |\cdot|)$.

\begin{theorem}(Cald\'eron-Zygmund type lemma \cite{DengHan2009,AimarTAMS})\label{caldzygdec}
Let $f$ be a nonnegative integrable function and $\lambda>0$ be any height. Then there exist functions $h_i$ and balls $\BB_i$ such that the sets $S_i=\{x\in \RR^{N+1}: h_i(x)\neq 0\}$ are pairwise disjoint and 
\begin{align*}
    &g=f-\sum_i h_i \in L^1\cap L^\infty\mbox{ and }||g||_{L^\infty}\leq \tilde D \lambda,\\
    &\int\limits_{\RR^{N+1}} h_i(x)\,dx=0,\\
    &S_i\subset \BB_i,\,\,\text{and}\,\,\sum_i |\BB_i|\leq \tilde D\lambda^{-1}||f||_{L^1},
\end{align*} where the constant $\tilde D$ only depends on the homogeneous space.
\end{theorem}

\section{Representation Formula}

\subsection{Fundamental solution}
The operator $\mathcal{L}_{0}$ is hypoelliptic, for example, see \cite{hormander67}. Let $\Gamma_0$ denotes its fundamental solution with pole at zero. We recall the following expression for $\Gamma_0$ as derived in \cite{hormander67,Kup72,LAnPol94}.

Let $A(z)=(a_{ij}(z))_{s_0\times s_0}$ denote the $s_0\times s_0$ matrix that features in \eqref{maineq}. Let $A_0$ denote the $N\times N$ constant matrix defined as
\[A_0:=
\begin{pmatrix}
  A(z_0)
   & \bigzero \\
  \bigzero & \bigzero
\end{pmatrix},
\] where each $\bigzero$ represents a block matrix with zeros.
Recall the definiton of the matrix polynomial $F(s)$ from \cref{translation}. Let
$$C(\tau, z_{0})=\int_{0}^{\tau}F(s)A_{0}F^{t}(s)\ ds.$$ Then, the fundamental solution of $\mathcal{L}_0$ with pole at zero is
\begin{align*}
\Gamma_0(z)=\Gamma_0(x, t) = \frac{1}{(4\pi)^{\frac N2}\left(\det C(t,z_0)\right)^{\frac 12}}\exp\left(-\frac 14 \langle C^{-1}(t,z_0)x,x\rangle\right) 
\end{align*}
The fundamental solution $\Gamma_0(\cdot,\zeta)$
of $L_0$ with pole in $\zeta\in \RR^{N+1}$ is defined as
$$
\Gamma_0(z,\zeta)=\Gamma_0(\zeta^{-1}\circ z), \qquad z, \zeta
\in \RR^{N+1},\quad z \neq \zeta.
$$

We collect here all the relevant properties of $\Gamma$ as proved in \cite{Bramanti-JMAA-ultraparabolic}.

\begin{theorem}(\cite[Theorem 2.2]{Bramanti-JMAA-ultraparabolic})\label{propGreen}
Let $z_0\in \RR^{N+1}$ and define $\Gamma(z_0,\cdot)\equiv \Gamma^0(\cdot)$. Then
\begin{itemize}
    \item[(i)] $\Gamma^0(\cdot)\in C^\infty(\RR^{N+1}\setminus\{0\})$.
    \item[(ii)] $\Gamma^0(\cdot)$ is $\delta_r$-homogeneous of degree $-Q$. 
    \item[(iii)] The function $\Gamma^0_i=\partial_{x_i}\Gamma^0$ is $\delta_r$-homogeneous of degree $-Q-a_i$ for $i=1,2,\ldots,N+1$.
    \item[(iv)] The function $\Gamma^0_{ij}=\partial_{x_ix_j}\Gamma^0$ is $\delta_r$-homogeneous of degree $-Q-a_i-a_j$ for $i=1,2,\ldots,N+1$.
    \item[(v)] The following estimates hold 
            \begin{align*}
                |\Gamma_i^0(z)|&\leq \frac{C}{||z||^{Q+a_i}},\nonumber\\
                |\Gamma_{ij}^0(z)|&\leq \frac{C}{||z||^{Q+a_i+a_j}},\mbox{ for }i,j=1,2,\ldots,N+1,
            \end{align*} where $\displaystyle C=\max\left\{ \sup_{S^N}|\Gamma_i^0(z)|, \sup_{S^N}|\Gamma_{ij}^0(z)|,i,j=1,2,\ldots,N+1\right\}$.
    \item[(vi)] For every $m\in \NN$, $z\in\RR^{N+1}$ and a multiindex $\beta$, it holds that
    \begin{align*}
        \sup_{||\zeta||=1,|\beta|=2m}\left|\left(\frac{\partial}{\partial \zeta}\right)^{\beta}\Gamma_{ij}^0(z;\zeta)\right|\leq C(m,\mu,B),
    \end{align*} for all $z\in\RR^{N+1}$.
    \item[(vii)] It holds that
    \begin{align*}
        \int\limits_{a<||\zeta||<b} \Gamma_{ij}^0(\zeta)d\zeta=0=\int\limits_{||\zeta||=1}\Gamma^0_{ij}(\zeta)\,d\sigma(\zeta),
    \end{align*} for $i,j=1,2,\ldots,s_0$ and $0<a<b$.
\end{itemize}
\end{theorem}


The second order derivatives of smooth $u$ have the following pointwise representation. 
\begin{theorem}(\cite[Theorem 2.4]{Bramanti-JMAA-ultraparabolic})
Let $u\in C_0^\infty(\RR^{N+1})$, $u=0$ for $t\leq 0$, $z\in\text{spt}\, u$. Then, it holds that for $i,j=1,2,\ldots,s_0$

\begin{align}\label{represent}
    u_{x_ix_j}(z)&=- \text{pv}\int\limits_{\RR^{N+1}} \Gamma_{ij}(z;\zeta^{-1}\circ z)\times\left( \sum_{h,k=1}^{s_0} [a_{hk}(z)-a_{hk}(\zeta)]u_{x_hx_k}(\zeta)+\mathcal{L}u(\zeta) \right)d\zeta\nonumber\\
    &\qquad\qquad-\mathcal{L}u(z)\cdot \int\limits_{||\zeta||=1}\Gamma_j(z;\zeta)\nu_id\sigma(\zeta),
\end{align}

where $\nu_i$ is the $i^{th}$ component of the outer normal to the surface $S^{N}$, the sphere in $\RR^{N+1}$.
\end{theorem}
We may write the expression in \cref{represent} in a more tractable manner as follows. We define the following:
\begin{align}\label{notat1}
    K_{ij}f(z)=\text{pv}\int\limits_{\RR^{N+1}}\Gamma_{ij}(z;\zeta^{-1}\circ z)f(\zeta)\,d\zeta,\nonumber\\
    \alpha_{ij}(z)=\int\limits_{||\zeta||=1}\Gamma_j(z;\zeta)\nu_id\sigma(\zeta),\mbox{ and }
\end{align}
for an operator $K$ and a function $a\in L^\infty(\RR^{N+1})$, define the {\it commutator} as
\begin{align}\label{notat2}
   [a, K](f)=K(af)-a\cdot K(f).
\end{align}

Using the notations \cref{notat1} and \cref{notat2} in \cref{represent}, we obtain the following compact expression for the second order derivatives of $u$.

\begin{align}\label{represent2}
    u_{x_ix_j}=- K_{ij}(\mathcal{L}u) + \sum_{h,k=1}^{s_0} [a_{hk}, K_{ij}](u_{x_hx_k})+\alpha_{ij}\cdot\mathcal{L}u,
\end{align} for $i,j=1,2,\ldots,s_0$.

In order to study the $L^p$ boundedness properties of the operators $K_{ij}$ and $[a_{hk}, K_{ij}]$, we will make use of spherical harmonics to obtain a Fourier series representation for the singular kernel $\Gamma_{ij}$. This technique was first used in the classic paper of Cald\'eron and Zygmund \cite{CaldZyg57}. In the case of anisotropic homogeneities, spherical harmonics found use in \cite{FabesRiv66}.

\begin{definition}(Spherical Harmonics)
A harmonic polynomial in $\RR^{N+1}$ which is homogeneous of degree $m$ is called a solid harmonic of degree $m$. The restriction of a solid harmonic to a sphere $S^N$ is called a spherical harmonic of degree $m$.
\end{definition}

The dimension of the space of spherical harmonics of degree $m$ is 
\begin{align}\label{dimSphHar}
    g_m=\begin{pmatrix}
m+N\\
N
\end{pmatrix}-\begin{pmatrix}
m+N-2\\
N
\end{pmatrix}\leq C_N \cdot m^{N-1}
\end{align}

The orthonormal system of spherical harmonics $\{Y_{km}\}_{k=1,\ldots,g_m}^{m=0,1,2,\ldots}$ is complete in $L^2(S^N)$.

We list below some properties of the spherical harmonics which may be found in \cite{CaldZyg57} and \cite{FabesRiv66}:
\begin{enumerate}
\item We have 
\begin{align*}\left|\left(\frac{\partial}{\partial x}\right)^\beta Y_{km}(x)\right|\leq C_N\cdot m^{\frac{N-1}{2}+|\beta|}, \mbox{ for } x\in S^{N}, k=1,2,\ldots, g_m.\end{align*}
\item Let $f\in C^\infty(S^N)$.  Let $f(x)\sim \sum_{k,m} b_{km}Y_{km}$ be its expansion in Fourier series, so that $b_{km}=\int\limits_{S^N}f(x)Y_{km}(x)\,d\sigma$, then for every $l\in\NN$, there is a constant $c$ depending on $l$ such that \begin{align*}|b_{km}|\leq c\cdot m^{-2l}\sup_{\stackrel{|\beta|=2l}{x\in S^N}}\left|\left(\frac{\partial}{\partial x}\right)^\beta f(x)\right|.\end{align*}
\item For a fixed $z\in\RR^{N+1}$ and $\zeta\in S^N$, we obtain the Fourier series expansion 
\begin{align*}\Gamma_{ij}(z;\zeta)=\sum_{m=0}^\infty\sum_{k=1}^{g_m} c_{ij}^{km}(z) Y_{km}(\zeta),\end{align*} and then for a general $\zeta\in \RR^{N+1}$, using the dilation $\zeta'=D(||\zeta||^{-1})\zeta$ and exploiting the homogeneity property of $\Gamma_{ij}$ we receive \begin{align*}\Gamma_{ij}(z;\zeta)=\sum_{m=0}^\infty\sum_{k=1}^{g_m} c_{ij}^{km}(z) \frac{Y_{km}(\zeta')}{||\zeta||^{Q+2}}.\end{align*}
\item $c_{ij}^{km}=0$ for $m=0$.
\item For $m>0$, we have the bound \begin{align}\label{riemlebdecay}|c_{ij}^{km}(z)|\leq C\cdot m^{-2l} \mbox{ for any } l>1, z\in\RR^{N+1}.\end{align}
\end{enumerate}

Now, define \begin{align}
    K_{km}(z)=\frac{Y_{km}(z')}{||z||^{Q+2}}.
    \end{align}
    
    \begin{theorem}(\cite{Bramanti-JMAA-ultraparabolic})\label{propertieskern}
        The kernels $K_{km}$ satisfy the following properties:
\begin{enumerate}
    \item $K_{km}\in C^\infty(\RR^{N+1}\setminus\{0\})$.
    \item $K_{km}$ is $\delta_r$-homogeneous of degree $-(Q+2)$.
    \item $K_{km}$ satisfies the following growth condition, for any $z\in\RR^{N+1}\setminus\{0\}$,
                \begin{align}\label{kernelprop3}
                    |K_{km}(z)|\leq \frac{c_{km}}{||z||^{Q+2}}\mbox{ with } c_{km}\leq C_N\cdot m^{\frac{N-1}{2}}.
                \end{align}
    \item $K_{km}$ satisfies the following vanishing property
                \begin{align}\label{kernelprop4}
                    \int\limits_{||\zeta||=1} K_{km}(\zeta)\,d\sigma(\zeta)=0.
                \end{align}
\end{enumerate}
    \end{theorem}
    
We shall also require the following property of $K_{km}$.

\begin{theorem}(H\"ormander's condition \cite[Proposition 3.4]{Bramanti-JMAA-ultraparabolic})\label{horpropthm}
There exists $\beta=\beta(r)\in (0,1]$, $M=M(B)>1$, $c_{km}=c(N)\cdot m^{\frac{N+1}{2}}$ such that
\begin{align}\label{horprop1}
    |K_{km}(\eta^{-1}\circ \zeta)-K_{km}(\eta^{-1}\circ z)|\leq c_{km}\frac{||\zeta^{-1}\circ z||^{\beta}}{||\eta^{-1}\circ z||^{Q+2+\beta}}, and
\end{align}
\begin{align}\label{horprop2}
    |K_{km}(\zeta^{-1}\circ \eta)-K_{km}(z^{-1}\circ \eta)|\leq c_{km}\frac{||\zeta^{-1}\circ z||^{\beta}}{||\eta^{-1}\circ z||^{Q+2+\beta}},
\end{align} for every $z,\zeta,\eta\in\RR^{N+1}$ satisfying $||\eta^{-1}\circ z||\geq M||\zeta^{-1}\circ z||$.
\end{theorem}

\begin{remark}
By an abuse of notation, we use the same factor $c_{km}$ in \cref{kernelprop3} and \cref{horpropthm} since in our estimates, we will always majorize by the larger number.
\end{remark}

Now, we define a further singular operator
\begin{align}
    T_{km}f(z)=\text{p.v.}\int\limits_{\RR^{N+1}} K_{km}(\zeta^{-1}\circ z)f(\zeta)\,d\zeta,
\end{align} so that \cref{represent2} unfolds to

\begin{align}\label{represent3}
    u_{x_ix_j}=- \sum_{m=1}^\infty\sum_{k=1}^{g_m}c_{ij}^{km}T_{km}(\mathcal{L}u) + \sum_{h,k=1}^{s_0}  \sum_{m=1}^\infty\sum_{k=1}^{g_m} c_{ij}^{km}\lbrack a_{hk}, T_{km}\rbrack(u_{x_hx_k})+\alpha_{ij}\cdot\mathcal{L}u,
\end{align} for $i,j=1,2,\ldots,s_0$.

\section{Pointwise estimates}

\subsection{Sparse Domination Preliminaries}\label{sparsedomprop}
We begin with the uncentered Hardy-Littlewood maximal function is defined as follows
$$\displaystyle\mathcal{M}f(\zeta)=\sup_{\mathcal{B}\ni \zeta}\fint\limits_{\mathcal{B}} |f(z)|\,dz,$$
where the supremum is over all balls $\mathcal{B}$ containing the points $\zeta$. 
Another useful variant of the \HLM~is the dyadic \HLM~and to define that we need the following notion of dyadic grids for homogeneous spaces, see \cite{HytKai}.

Let $0<c_1\leq C_1<\infty$ and $\delta\in (0,1)$. By a general dyadic grid $\mathscr{U}=\bigcup_{k\in \mathbb{Z}}\mathscr{U}_k$ on $\RR^{N+1}$, we mean a countable collection of sets $\mathscr{U}_k^{\alpha}$ for $k\in\ZZ$, each associated with a point $z_k^\alpha$, $\alpha$ coming from a countable index, with the following properties:
\begin{itemize}
    \item $\RR^{N+1}=\bigcup\limits_{\alpha}\mathscr{U}_k^{\alpha}$ for every $k\in\ZZ$.
    \item If $l\geq k$, then either $\mathscr{U}_l^{\beta}\subset \mathscr{U}_k^{\alpha}$ or $\mathscr{U}_l^{\beta}\cap \mathscr{U}_k^{\alpha}=\emptyset$.
    \item There exist $c_1, C_1>0$ such that $B(z_k^\alpha,c_1\delta^k)\subset \mathscr{U}_k^\alpha\subset B(z_k^\alpha,C_1\delta^k):=B(\mathscr{U}_k^\alpha)$.
    \item If $l\geq k$ and $\mathscr{U}_l^{\beta}\subset \mathscr{U}_k^{\alpha}$, then $B(\mathscr{U}_l^{\beta})\subset B(\mathscr{U}_k^{\alpha})$.
    \item For all $k\in\ZZ$, there exists $\alpha$ such that $x_0=z_k^\alpha$.
\end{itemize}
Hyt\"onen and Kairema \cite[Theorem 4.1]{HytKai} proved the existence of a finite collection of dyadic grids $\mathscr{U}^t$, $t=1,2,\ldots,L$ such that for every ball $B(z,r)\subset \RR^{N+1}$ with $\delta^{k+2}\leq r<\delta^{k+1}$, there exists some $t\in \{1,2,\ldots,L\}$ and $\mathscr{U}_k^\alpha\in \mathscr{U}^t$ such that $B(z,r)\subset \mathscr{U}_k^\alpha$ and $\text{diam}\,\mathscr{U}_k^\alpha\leq C\,r$, where $C$ depends on $\delta$.

Corresponding to a dyadic grid $\mathscr{U}$, the dyadic \HLM~is defined as follows:
$$\displaystyle\mathcal{M}_{\mathscr{U}}f(\zeta)=\sup_{\PP\ni \zeta; \PP\in \mathscr{U} }\fint\limits_{\PP} |f(z)|\,dz.$$
It is obvious that $\mathcal{M}_{\mathscr{U}}$ is pointwise controlled by the \HLM~$\mathcal{M}$, but more interestingly, as proved in \cite[Prop. 7.9]{HytKai}, we have
\begin{align}
\label{comparison-dyadic-full}
\mathcal{M}f\lesssim \sum_{t=1}^{L} \mathcal{M}_{\mathscr{U}^t}f.
\end{align}
In the rest of the paper, an element of a dyadic grid as above will be called as an ultraparabolic dyadic cube $\PP\in \UU$. We reserve the notation $\UU(\PP)$ for the collection of all dyadic cubes $\TT\in \UU$ such that $\TT\subset \PP$ and $\displaystyle \mathcal{M}^{\PP}f(\zeta):=\sup_{\TT\ni \zeta; \TT\in \mathscr{U}(\PP) }\fint\limits_{\TT} |f|$ is the restricted \HLM. 



\begin{definition}[Sparse family]
Let $\mathscr{G}$ be a family of cubes belonging to some general dyadic grid and $\eta\in (0, 1)$. We call $\mathscr{G}$ to be a {\it{$\eta$-sparse family}} if for every $\mathcal{P}\in \mathscr{G}$ there exists a set $F_{\PP}$ such that $|F_{\PP}|\geq \eta |\PP|$ and $\{F_{\PP}\}_{\GG}$ are mutually disjoint.
\end{definition}

\begin{remark}
In application we shall work with collection of ultraparabolic dyadic cubes $\GG\subset \UU$ such that for any $\TT\in \GG$ we have 
\begin{equation}
\sum_{\TT'\in \UU: \TT'\in \UU(\TT)}|\TT'|\leq \frac{1}{2}|\TT|.
\label{sparse-cond}
\end{equation}
The above condition ensures that the collection $\GG$ is sparse since for each $\TT\in \GG$ if we define \[\displaystyle F_{\TT}=\TT\setminus\bigcup_{\TT'\in \UU(\TT)}\TT',\] then by the dyadic structure of $\UU$, the collection $\{F_{\TT}\}_{\GG}$ is pairwise disjoint and \eqref{sparse-cond} implies \[\displaystyle |F_{\TT}|\,\geq\, |T|\quad -\sum_{\TT'\in \UU: \TT'\in \UU(\TT)}|\TT'|\,\geq\, \frac{1}{2}|\TT|.\]
\label{sparse-remark}
\end{remark}

\begin{definition}[Sparse operator]
\label{defition-sparse-operator}
Given a {\it sparse family} $\GG$, the associated sparse operator $\mathscr{A}_{\GG}$ is defined as follows
\begin{align*}
\mathscr{A}_{\GG}f(x):=\sum_{\PP\in \GG}\left(\frac{1}{|\PP|}\int_{\PP}|f|\right)\chi_{\PP}(x).    
\end{align*}
We also introduce the following sparse operators in order to estimate the commutator terms appearing in \eqref{represent3}. Corresponding to a sparse family $\GG$ and a function $b$, the operator $\mathscr{A}_{\GG, b}$ is defined as follows
\begin{align*}
\mathscr{A}_{\GG}^{ b}f(x):=\sum_{\PP\in \GG}|b(x)-b_{\PP}|\left(\frac{1}{|\PP|}\int_{\PP}|f|\right)\chi_{\PP}(x).    
\end{align*}
The adjoint of $\mathscr{A}_{\GG}^{ b}$ is given by 
\begin{align*}
\mathscr{A}_{\GG}^{ b, *}f(x):=\sum_{\PP\in \GG}\left(\frac{1}{|\PP|}\int_{\PP}|b-b_{\PP}||f|\right)\chi_{\PP}(x).    
\end{align*}
\end{definition}

We let $\mathcal{M}$ stand for the uncentered Hardy--Littlewood maximal function on $(\mathbb{R}^{N+1}, d, | \cdot |)$ and for $1<r<\infty$, $\mathcal{M}_{r}f:=(\mathcal{M}(|f|^r))^{1/r}$. Our proofs on sparse domination depend on a general sparse domination principle by Lerner--Ombrosi, see Theorem 1.1 in \cite{Lerner-Ombrosi-pointwaise-sparse2020}. For any linear operator $T$, let us consider the following version of the grand maximal truncated operator: For $s > 0$
\begin{align} \label{def:grand-maximal-truncated-operator}
\mathcal{M}^{\#}_{T, s} f(x) = \sup_{B: B \ni x} \esssup \limits_{y,z \in B} \left| T(f \chi_{\mathbb{R}^{N+1}\setminus s B})(y)-T(f \chi_{\mathbb{R}^{N+1}\setminus s B})(z) \right|, 
\end{align}
where the supremum is taken over all balls $B$ containing the point $x$.

Let us state the following pointwise result which is even new in the context of elliptic and parabolic equations.

\begin{theorem}\label{represent_pointwise_bound}
Suppose that the matrices $A$ and $C$ satisfy \descref{H1}{H1}and \descref{H2}{H2}. Fix
$i,j\in\{1,2,\ldots,s_0\}$.
Then for any ultraparabolic cube $\PP\subset \Omega$, there exist sparse families $\{\GG_{km, l}: k=1, \cdots, g_{m}, 1\leq l\leq C(N), m\in \mathbb{N}\}$ such that for any $u\in \mathcal{S}^X(\Omega,\mathcal{L})$, we have
\begin{align}\label{represent3-sparse}
    |u_{x_ix_j}(z)|&\leq \kappa_{A, N, C} \sum_{m=1}^\infty\sum_{k=1}^{g_m}c_{km} |c_{ij}^{km}| \sum_{l=1}^{C(N)}\,  \mathscr{A}_{\GG_{km, l}}(\mathcal{L}u)(z) \nonumber\\
    &\quad+  \sum_{h,k=1}^{s_0}  \sum_{m=1}^\infty\sum_{k=1}^{g_m}c_{km} |c_{ij}^{km}| \sum_{l=1}^{C(N)} \mathscr{A}_{\GG_{km, l} }^{a_{hk}}(u_{x_hx_k})(z)\nonumber\\
    &\quad\quad+\sum_{h,k=1}^{s_0}  \sum_{m=1}^\infty\sum_{k=1}^{g_m}c_{km} |c_{ij}^{km}| \sum_{l=1}^{C(N)} \, \mathscr{A}_{\GG_{km, l} }^{a_{hk}, *}(u_{x_hx_k})(z)+|\alpha_{ij}|\cdot|\mathcal{L}u(z)|,
\end{align} for a.e. $z\in \PP$. 
\end{theorem}

To conclude the proof of the above theorem we need to individually handle each term appearing in \eqref{represent3}. For the rest of this section fix some $k, m$ such that $m\in \mathbb{N}$ and $k=1, \cdots, g_{m}$.

\subsection{End-point boundedness }

In order to kickstart the sparse domination method, we will require a weak-$L^1$ type estimate for the operator $T_{km}$. This, in turn, requires the $L^2$ boundedness of the operator $T_{km}$. The proof of $L^2$ boundedness can be obtained from the proof given in Aimar\cite{AimarTAMS}.

\begin{theorem}\cite[Theorem 3.6]{Bramanti-JMAA-ultraparabolic}($L^2$ boundedness of $T_{km}$)\label{l2bounded}
Let us define 
\begin{align*}
    T_{km}f(z)=\text{p.v.}\int\limits_{\RR^{N+1}} K_{km}(\zeta^{-1}\circ z)f(\zeta)\,d\zeta.
\end{align*} Then
\begin{align*}
    ||T_{km}f||_{L^2}\leq c_{km}||f||_{L^2},
\end{align*} where $c_{km}$ is the constant that appears in \cref{horpropthm}.
\end{theorem}

\begin{theorem}(Weak-$L^1$ type estimates for $T_{km}$)
Let us define 
\begin{align*}
    T_{km}f(z)=\text{p.v.}\int\limits_{\RR^{N+1}} K_{km}(\zeta^{-1}\circ z)f(\zeta)\,d\zeta.
\end{align*} Then $T_{km}$ maps $L^{1}(\RR^{N+1}, d, |\cdot|)$ to $L^{1, \infty}(\RR^{N+1}, d, |\cdot|)$ boundedly, that is, for all $\lambda>0$
\begin{align*}
    |\{z\in \RR^{N+1}: |T_{km}f(z)|>\lambda\}|\leq \frac{c_{km}}{\lambda}\int\limits_{\RR^{N+1}}|f(z)|\,dz,
\end{align*} where $c_{km}$ is the constant that appears in \cref{horpropthm}.
\label{weak-type}
\end{theorem}

\begin{proof}
   Let $f\in L^1(\RR^{N+1}, d, |\cdot|)$ and $\lambda>0$ be fixed. We apply the Cald\'eron-Zygmund type decomposition adapted for homogeneous spaces (as in \cref{caldzygdec}) to the function $f$ and level $\lambda$. In consequence, we obtain 
   \begin{align}\label{decomp1}f=g+\sum h_i\mbox{ where }||g||_{L^\infty}\leq \lambda,\end{align} the functions $h_i$ are supported in balls $\BB_i$ whose total measure is controlled by $\sum|\BB_i|\leq \tilde{D}\lambda^{-1}||f||_{L^1}$.
   
   The set 
   \begin{align*}
       \{z\in\RR^{N+1}:T_{km}f(z)>\lambda\}
   \end{align*} is contained in the union of the sets
    \begin{align*}
       \underbrace{\left\{z\in\RR^{N+1}:T_{km}g(z)>\frac{\lambda}{2}\right\}}_{G}\mbox{ and }\underbrace{\left\{z\in\RR^{N+1}:T_{km}\left(\sum h_i\right)(z)>\frac{\lambda}{2}\right\}}_{B},
   \end{align*} hence, it suffices to estimate their measures. For the first of these sets $G$, we use \cref{l2bounded}. The estimate of the good set is as follows:
   \begin{align}
       \left|\left\{z\in\RR^{N+1}:T_{km}g(z)>\frac{\lambda}{2}\right\}\right|&\overset{\redlabel{cheby}{a}}{\leq} \frac{C}{\lambda^2}||T_{km}g||^2_{L^2}\nonumber\\
       &\overset{\redlabel{l2b}{b}}{\leq} \frac{c_{km}}{\lambda^2}||g||_{L^2}^2\nonumber\\
       &=\frac{c_{km}}{\lambda^2}\int\limits_{\RR^{N+1}}|g(z)|^2\,dz\nonumber\\
       &\overset{\redlabel{gbound}{c}}{\leq} \frac{c_{km}}{\lambda}\int\limits_{\RR^{N+1}}|g(z)|\,dz\nonumber\\
       &\leq \frac{c_{km}}{\lambda}\int\limits_{\RR^{N+1}}|f(z)|\,dz,\label{l1est1}
   \end{align} where \redref{cheby}{a} follows from the Chebyshev's inequality, \redref{l2b}{b} follows from \cref{l2bounded}, \redref{gbound}{c} is a consequence of \cref{decomp1}.
   
   To estimate the measure of the remaining set $B$, we begin by excising the supports of $h_i$ from $B$ since they are already controlled in the desired way. Indeed, in the ensuing calculation, this removal is crucial to create a gap between points in $\BB_i$ and their complements which allows the application of \cref{horpropthm}. 
   
   Observe that due to the doubling property,
   \begin{align}\label{l1est2}
       \left|\bigcup_i 2\BB_i\right|\leq C\sum_i|\BB_i|\leq \frac{C}{\lambda}||f||_{L^1}. 
   \end{align}
   Therefore, instead of $B$ we look at the new set
   \begin{align*}
       \left\{z\in\RR^{N+1}\setminus \bigcup_i 2\BB_i:T_{km}\left(\sum h_i\right)(z)>\frac{\lambda}{2}\right\}.
   \end{align*}
   Once more with an application of Chebyshev's inequality, we have
      \begin{align}
       \left|\left\{z\in\RR^{N+1}\setminus \bigcup_i 2\BB_i:T_{km}\left(\sum h_i\right)(z)>\frac{\lambda}{2}\right\}\right|&\leq \frac{C}{\lambda}\int\limits_{\RR^{N+1}\setminus \bigcup_i 2\BB_i} \left|\sum_{i} T_{km}h_i(z)\right|\,dz\nonumber\\
       &\leq \frac{C}{\lambda}\sum_{i}\underbrace{\int\limits_{\RR^{N+1}\setminus \bigcup_i 2\BB_i} \left|T_{km}h_i(z)\right|\,dz}_{E_i}\label{l1est3}
   \end{align}
   Let the ball $\BB_i$ be centered at $z_i\in\RR^{N+1}$ and have radius $r_i$, then
   \begin{align}
       E_i&=\int\limits_{\RR^{N+1}\setminus \bigcup_i 2\BB_i} \left|T_{km}h_i(z)\right|\,dz\nonumber\\
       &=\int\limits_{\RR^{N+1}\setminus \bigcup_i 2\BB_i} \left|\int\limits_{\BB_i} K_{km}(\zeta^{-1}\circ z)h_i(\zeta)\,d\zeta\right|\,dz\nonumber\\
       &\overset{\redlabel{vanish}{a}}{=}\int\limits_{\RR^{N+1}\setminus \bigcup_i 2\BB_i} \left|\int\limits_{\BB_i} \left(K_{km}(\zeta^{-1}\circ z)-K_{km}(z_i^{-1}\circ z)\right)h_i(\zeta)\,d\zeta\right|\,dz\nonumber\\
       &\leq \int\limits_{\RR^{N+1}\setminus 2\BB_i} \int\limits_{\BB_i}\left|K_{km}(\zeta^{-1}\circ z)-K_{km}(z_i^{-1}\circ z)\right||h_i(\zeta)|\,d\zeta\,dz\nonumber\\
       &\overset{\redlabel{horprop}{b}}{\leq} c_{km}\int\limits_{\RR^{N+1}\setminus 2\BB_i} \int\limits_{\BB_i}\frac{||\zeta^{-1}\circ z_i||^{\beta}}{||z^{-1}\circ z_i||^{Q+2+\beta}}|h_i(\zeta)|\,d\zeta\,dz\nonumber\\
       &\overset{\redlabel{fubini}{c}}{=} c_{km} \int\limits_{\BB_i}\left(\int\limits_{\RR^{N+1}\setminus 2\BB_i}\frac{||\zeta^{-1}\circ z_i||^{\beta}}{||z^{-1}\circ z_i||^{Q+2+\beta}}\,dz\right)|h_i(\zeta)|\,d\zeta\nonumber\\
       &= c_{km} \int\limits_{\BB_i}\left(\sum_{k=1}^\infty\int\limits_{2^{k+1}\BB_i\setminus 2^{k}\BB_i}\frac{||\zeta^{-1}\circ z_i||^{\beta}}{||z^{-1}\circ z_i||^{Q+2+\beta}}\,dz\right)|h_i(\zeta)|\,d\zeta\nonumber\\
       &\leq  c_{km} \int\limits_{\BB_i}\left(\sum_{k=1}^\infty\int\limits_{2^{k+1}\BB_i\setminus 2^{k}\BB_i}\frac{r_i^{\beta}}{(2^k r_i)^{Q+2+\beta}}\,dz\right)|h_i(\zeta)|\,d\zeta\nonumber\\
       &\leq  c_{km} \int\limits_{\BB_i}\left(\sum_{k=1}^\infty\frac{r_i^{\beta}(2^k r_i)^{Q+2}}{(2^k r_i)^{Q+2+\beta}}\right)|h_i(\zeta)|\,d\zeta\nonumber\\
       &\leq  c_{km} \int\limits_{\BB_i}\left(\sum_{k=1}^\infty\frac{1}{(2^\beta)^k}\right)|h_i(\zeta)|\,d\zeta\leq c_{km} \int\limits_{\BB_i}|h(\zeta)|\,d\zeta,\label{l1est4}
   \end{align} where \redref{vanish}{a} is due to the fact that $\int h_i=0$, \redref{horprop}{b} is due to \cref{horprop2}, and \redref{fubini}{c} is due to Fubini's theorem. 
   
   
   Using \cref{l1est4} in \cref{l1est3} results in the following estimate
    \begin{align}
       \left|\left\{z\in\RR^{N+1}\setminus \bigcup_i 2\BB_i:T_{km}\left(\sum h_i\right)(z)>\frac{\lambda}{2}\right\}\right|&\leq \frac{C}{\lambda}\sum_{i} E_i\nonumber\\
       &\lesssim \frac{c_{km}}{\lambda}\sum_{i}\int\limits_{\BB_i}|h_i(\zeta)|\,d\zeta\nonumber\\
       &\lesssim \frac{c_{km}}{\lambda}\int\limits_{\RR^{N+1}}|f(\zeta)|\,d\zeta.\label{l1est5}
   \end{align}
   Now, the weak-$L^1$-type estimate is a consequence of a combination of the three estimates in \cref{l1est1}, \cref{l1est2}, and \cref{l1est5}.
\end{proof}

\begin{theorem}(weak $L^1$-type estimate for $\mathcal{M}_{T_{km},s}^{\#}$)
\label{weak-type-sharp}
Let $s>1$, then the grand truncated maximal function $\mathcal{M}_{T_{km},s}^{\#}$ maps $L^{1}(\RR^{N+1}, d, |\cdot|)$ to $L^{1, \infty}(\RR^{N+1}, d, |\cdot|)$ boundedly.
\end{theorem}

\begin{proof}
   We recall the definition of the grand truncated maximal function. 
   \begin{align*}
       \mathcal{M}^{\#}_{T_{km}, s} f(\zeta) = \sup_{\BB: \BB \ni \zeta} \esssup \limits_{y,z \in \BB} \left| T(f \chi_{\mathbb{R}^{N+1}\setminus s \BB})(y)-T(f \chi_{\mathbb{R}^{N+1}\setminus s \BB})(z) \right|.
   \end{align*} For the proof, we will estimate the grand truncated maximal function by the uncentered Hardy-Littlewood maximal function $\mathcal{M}$ in a pointwise manner, and this will conclude the proof of the Theorem since $\mathcal{M}$ is weak type $(1, 1)$.
   
   Fix a point $\zeta\in\RR^N$ and choose a ball $\BB\ni \zeta$ of radius $r$. Let $y,z\in \BB$ and consider the expression
   \begin{align*}
       \left|T_{km}(f\chi_{\RR^{N+1}\setminus s\BB})(y)-T_{km}(f\chi_{\RR^{N+1}\setminus s\BB})(z)\right|&=\left|\int\limits_{\RR^{N+1}\setminus s\BB} (K_{km}(\eta^{-1}\circ y)-K_{km}(\eta^{-1}\circ z))f(\eta)\,d\eta\right|\\
       &\leq\int\limits_{\RR^{N+1}\setminus s\BB} \left|K_{km}(\eta^{-1}\circ y)-K_{km}(\eta^{-1}\circ z)\right||f(\eta)|\,d\eta\\
       &\leq c_{km}\int\limits_{\RR^{N+1}\setminus s\BB}\frac{||y^{-1}\circ z||^{\beta}}{||\eta^{-1}\circ z||^{Q+2+\beta}}|f(\eta)|\,d\eta\\
       &= c_{km}\sum_{k=1}^\infty\int\limits_{s^{k+1}\BB\setminus s^k\BB}\frac{||y^{-1}\circ z||^{\beta}}{||\eta^{-1}\circ z||^{Q+2+\beta}}|f(\eta)|\,d\eta\\
       &\leq c_{km}\sum_{k=1}^\infty\int\limits_{s^{k+1}\BB\setminus s^k\BB}\frac{(r)^{\beta}}{(s^k r)^{Q+2+\beta}}|f(\eta)|\,d\eta\\
       &\leq c_{km}\sum_{k=1}^\infty\frac{(r)^{\beta}(s^k r)^{Q+2}}{(s^k r)^{Q+2+\beta}}\fint\limits_{s^{k+1}\BB}|f(\eta)|\,d\eta\\
       &\leq \tilde{c}_{km}\fint\limits_{s^{k+1}\BB}|f(\eta)|\,d\eta.
   \end{align*}
   Now, taking supremum over all $y,z\in\BB$ followed by supremum over all $\BB\ni \zeta$ gives us
   \begin{align*}
       \mathcal{M}^{\#}_{T_{km}, s} f(\zeta) \leq c_{km}\mathcal{M}f(\zeta).
   \end{align*}
   
\end{proof}
\subsection{Local pointwise estimates}

Now we are in a position to prove the following local sparse domination result. The following result essentially proves the poitwise domination of commutators by sparse operators. The proof is quite involved and probably is of independent interest due to its applicability for a large class of operators. For any ultraparabolic cube $\PP$, $\PP^*$ denotes the dilated cube  $\delta_{c_{N}}\PP$, where $c_{N}$ is a fixed dimensional constant to be chosen later.

\begin{theorem}[Pointwise estimates for the commutator]
\label{Pointwise estimates-comm}
Let $\PP$ be any ultraparabolic cube. Then for any compactly supported bounded function $f$ there exists a $\frac{1}{2}$-sparse family $\GG\subset \mathscr{U}(\PP)$ such that

\begin{align}
\label{pointwise-sparse-local-comm}
\big|[b, T_{km}](f \chi_{\PP^*})(z)\big|\lesssim_{N}\,  c_{km}\,& \left[\sum\limits_{\TT\in \GG}|b-b_{\TT^*}|\left(\fint_{\TT^{*}}|f|\right) \chi_{\TT}(z)+\sum\limits_{\TT\in \GG}\left(\fint_{\TT^{*}}|(b-b_{\TT^{*}})f|\right) \chi_{\TT}(z)\right]\\
&\qquad+ c_{km} \|b\|_{\text{BMO}} \sum\limits_{\TT\in \GG}\left(\frac{1}{|\TT^*|}\int_{\TT^{*}}|f|\right) \chi_{\TT}(z)\notag   
\end{align}
holds for a.e. $z\in \PP$.
\end{theorem}
\begin{proof}
For simplicity let us denote $\displaystyle\ev{f}_{\TT^*}=\fint_{\TT^{*}}|f|\,dz$. Define the following operator 
 \begin{align*}
\mathcal{N}_{\PP}f(z):=\max\{|T_{km}(f \chi_{\PP^*})(z)|, \mathcal{M}^{\sharp}_{T_{km}, s}(f \chi_{\PP^*})(z)\}.
 \end{align*}
Theorem~\ref{weak-type} and Theorem~\ref{weak-type-sharp} together imply that $\mathcal{N}$ is weak-type $(1, 1)$. Let us consider the following sets 
\begin{align*}
&F^{1}_{\PP,\,bad}:=\left\{ z\in \PP: \max\left\{\mathcal{M}f(z), \mathcal{N}(f\chi_{\PP^*})(z)\right\}> \alpha_{N}\ c_{km} \langle f \rangle_{\PP^*}\right\},\\
& F^{2}_{\PP,\,bad}:=\left\{z\in \PP: \mathcal{N}((b-b_{\PP^*})f)(z)> \alpha_{N}\ c_{km}  \langle (b-b_{\PP^*})f \rangle_{\PP^*} \right\}.
\end{align*}
Denote $F_{\PP,\,bad}=F^{1}_{\PP,\,bad}\cup F^{2}_{\PP,\,bad}$. Since $\mathcal{M}$ and $\mathcal{N}$ are weak-type $(1, 1)$, it is possible to choose large enough $\alpha_N$ such that
\begin{align}
|F_{\PP,\,bad}|\lesssim O\left(\frac{1}{\alpha_N}\right)|\PP|\leq \frac{1}{2}|\PP|.    
\end{align}
Note that we have used the fact $|\PP^*|\lesssim |\PP|$ in the above estimate. To construct appropriate sparse family, we shall decompose the set $F_{\PP,\,bad}$ using \CZ~decomposition. To that end, let us consider
\begin{align}
F_{\lambda, \mathcal{P}}:=\{z\in \PP: \mathcal{M}^{\PP}(\chi_{F_{\PP,\,bad}})(z)>\lambda\}.
\end{align}
Local \CZ~decomposition for small enough $\lambda$ produces a collection of pairwise disjoint ultraparabolic cubes $\GG_{\PP}\subset \UU(\PP)$ such that $F_{\PP, bad}\subset F_{\lambda, \PP}=\cup_{\TT\in \GG_{\PP}}\TT$ and
\begin{align}
\lambda<\frac{|\TT \cap F_{\PP, bad}|}{|\TT|}\leq \frac{1}{2},    
\end{align}
for all $\TT\in \GG_{\PP}$. The above implies 
\begin{align}\sum\limits_{\TT\in \GG_{\PP}} |\TT|\leq \frac{1}{\lambda} \sum\limits_{\TT\in \GG_{\PP}} |\TT \cap F_{\PP,\,bad}|\lesssim \frac{1}{\lambda} |F_{\PP,\,bad}| \lesssim O\left(\frac{1}{\alpha_N \lambda }\right)|\PP| \leq \frac{1}{2} |\PP|,
\label{sparseness-comm}
\end{align}
provided we choose $\alpha_N$ very large.

Define $\GG_{0}=\{\PP\}$ and $\GG_{1}=\GG_{\PP}$. For each $\TT\in \GG_{\PP}$, similar procedure generates the collection $\GG_{\TT}$ and we combine them to define $\GG_{2}=\bigcup_{\TT\in \GG_{\PP}} \GG_{\TT}$. Iteratively we define the collection $\GG_{k}$ for all $k=0, 1, 2, \ldots$ Therefore, from Remark~\ref{sparse-remark} and \eqref{sparseness-comm} it follows that the collection $\GG=\cup_{k\geq 0} \GG_{k}$ is a sparse collection.

Rest of the proof is dedicated to conclude that $\GG$ is a required collection for which \eqref{pointwise-sparse-local-comm} holds true. The recursive process implies that $\sum_{\TT\in \GG_{k}}|\TT|\leq \frac{1}{2^k}|\PP|$, consequently, $|\GG_k|\to 0$ as $k\to \infty$. Therefore, up to a measure zero set, for each point $z\in \PP$ there exists a largest $k=k(z)\in\NN$ depending on $z$ such that $z\in \GG_{j}$ for each $0\leq j\leq k(z)$. Hence, there exist a chain of ultraparabolic dyadic cubes $\{\TT_{j}\}_{j=0}^{k(z)}$ with $\TT_{j}\in \GG_j$ such that $z\in \TT_{k(z)}\subset \TT_{k(z)-1}\subset\cdots\subset \TT_{0}=\PP$. Since $[b, T]f=[b-c, T]f$ for any constant $c\in \RR$, note that
\begin{align}
\nonumber\ek{[b, T_{km}](f \chi_{\PP^*})(z)}&\leq \sum_{j=1}^{k(z)}\ek{[b, T_{km}](f \chi_{\TT_{j-1}^*\setminus \TT_{j}^*})(z)}+\ek{[b, T_{km}](f \chi_{\TT^{*}_{k(z)}})(z)|}\\
&\nonumber = \sum_{j=1}^{k(z)}\ek{[b-b_{\TT_{j-1}^*}, T_{km}](f \chi_{\TT_{j-1}^*\setminus \TT_{j}^*})(z)}+\ek{[b-b_{\TT_{k(z)}^*}, T_{km}](f \chi_{\TT^{*}_{k(z)}})(z)}\\
&\nonumber \lesssim \sum_{j=1}^{k(z)} \left(\ek{b-b_{\TT_{j-1}^*}}\ek{T_{km}(f \chi_{\TT_{j-1}^*\setminus \TT_{j}^*})(z)}+\ek{T_{km}((b-b_{\TT_{j-1}^*})f \chi_{\TT_{j-1}^*\setminus \TT_{j}^*})(z)}\right)\\
&\nonumber\qquad  + \left(\ek{b-b_{\TT_{k(z)}^*}}\ek{T_{km}(f \chi_{\TT^{*}_{k(z)}})(z)}+\ek{T_{km}((b-b_{\TT_{k(z)}^*})f \chi_{\TT^{*}_{k(z)}})(z)}\right)\\
&\nonumber \lesssim \sum_{j=1}^{k(z)} \left(\underbrace{\ek{b-b_{\TT_{j-1}^*}} \ek{T_{km}(f \chi_{\TT_{j-1}^*\setminus \TT_{j}^*})(z)}}_{I_j}+\underbrace{\ek{T_{km}((b-b_{\TT_{j-1}^*})f \chi_{\TT_{j-1}^*\setminus \TT_{j}^*})(z)}}_{II_j}\right)\\
&\qquad\qquad+c_{km}\left(\ek{(b-b_{\TT_{k(z)}^*})}\ev{f}_{\TT^{*}_{k(z)}}+ \ev{(b-b_{\TT_{k(z)}^*})f}_{\TT^{*}_{k(z)}}\right),
\label{pointwise-recursive}
\end{align}
where in the last inequality we have used the fact that $z\notin \GG_{k(z)+1}$. 

Now, we analyze the summands $I_j$ and $II_j$ in the sum in \cref{pointwise-recursive} for $j\in\{1,2,\ldots,k(z)\}$. Note that for each $z'\in  \TT_{j}$, we can write 
\begin{align}\label{estIj}
I_j&=\ek{b-b_{\TT_{j-1}^*}} \ek{T_{km}(f \chi_{\TT_{j-1}^*\setminus \TT_{j}^*})(z)}\\
&\leq \ek{b-b_{\TT_{j-1}^*}}\ek{T_{km}(f \chi_{\TT_{j-1}^*\setminus \TT_{j}^*})(z)-T_{km}(f \chi_{\TT_{j-1}^*\setminus \TT_{j}^*})(z')}\nonumber\\
&\qquad+\ek{b-b_{\TT_{j-1}^*}}\ek{T_{km}(f \chi_{\TT_{j-1}^*\setminus \TT_{j}^*})(z')},\nonumber\\
&\leq \ek{b-b_{\TT_{j-1}^*}} \inf_{\zeta \in \TT_{j}} \mathcal{M}^{\sharp}_{T_{km}, c_{N}}(f \chi_{\TT_{j-1}^*})(\zeta)\nonumber\\
&\qquad+ \ek{b-b_{\TT_{j-1}^*}}\left(\inf_{\zeta \in \TT_{j}}\ek{T_{km}(f \chi_{\TT_{j-1}^*})(\zeta)}+ \inf_{\zeta \in \TT_{j}}\ek{T_{km}(f \chi_{\TT_{j}^*})(\zeta)}\right),\nonumber
\end{align} where the first inequality is obtained by adding and subtracting a suitable term.
Similarly, we have 

\begin{align}\label{estIIj}
II_j&=\ek{T_{km}((b-b_{\TT_{j-1}^*})f \chi_{\TT_{j-1}^*\setminus \TT_{j}^*})(z)}\\
&\leq\ek{T_{km}((b-b_{\TT_{j-1}^*})f \chi_{\TT_{j-1}^*\setminus \TT_{j}^*})(z)-T_{km}((b-b_{\TT_{j-1}^*})f \chi_{\TT_{j-1}^*\setminus \TT_{j}^*})(z')}\nonumber\\
&\qquad+\ek{T_{km}((b-b_{\TT_{j-1}^*})f \chi_{\TT_{j-1}^*\setminus \TT_{j}^*})(z')},\nonumber\\
&\leq \inf_{\zeta \in \TT_{j}} \mathcal{M}^{\sharp}_{T_{km}, c_{N}}((b-b_{\TT_{j-1}^*})f \chi_{\TT_{j-1}^*})(\zeta)\nonumber\\
&\qquad+\left(\inf_{\zeta \in \TT_{j}}\ek{T_{km}((b-b_{\TT_{j-1}^*})f \chi_{\TT_{j-1}^*})(\zeta)}+ \inf_{\zeta \in \TT_{j}}\ek{T_{km}((b-b_{\TT_{j-1}^*})f \chi_{\TT_{j}^*})(\zeta)}\right)\nonumber\\
&= \inf_{\zeta \in \TT_{j}} \mathcal{M}^{\sharp}_{T_{km}, c_{N}}((b-b_{\TT_{j-1}^*})f \chi_{\TT_{j-1}^*})(\zeta)\nonumber\\
&\qquad+\inf_{\zeta \in \TT_{j}}\ek{T_{km}((b-b_{\TT_{j-1}^*})f \chi_{\TT_{j-1}^*})(\zeta)}\nonumber\\
&\qquad\qquad+ \inf_{\zeta \in \TT_{j}}\ek{T_{km}((b-b_{\TT_{j}^*}+b_{\TT_{j}^*}-b_{\TT_{j-1}^*})f \chi_{\TT_{j}^*})(\zeta)}\nonumber\\
&\leq \inf_{\zeta \in \TT_{j}} \mathcal{M}^{\sharp}_{T_{km}, c_{N}}((b-b_{\TT_{j-1}^*})f \chi_{\TT_{j-1}^*})(\zeta)\nonumber\\
&\qquad+\inf_{\zeta \in \TT_{j}}\ek{T_{km}((b-b_{\TT_{j-1}^*})f \chi_{\TT_{j-1}^*})(\zeta)}\nonumber\\
&\qquad\qquad+\inf_{\zeta \in \TT_{j}}\ek{T_{km}((b-b_{\TT_{j}^*})f \chi_{\TT_{j}^*})(\zeta)}+\ek{b_{\TT_{j}^*}-b_{\TT_{j-1}^*}}\inf_{\zeta \in \TT_{j}} \ek{T_{km}(f\chi_{\TT^*_{j}}(\zeta)}.\nonumber\\
&\leq \inf_{\zeta \in \TT_{j}} \mathcal{M}^{\sharp}_{T_{km}, c_{N}}((b-b_{\TT_{j-1}^*})f \chi_{\TT_{j-1}^*})(\zeta)\nonumber\\
&\qquad+\inf_{\zeta \in \TT_{j}}\ek{T_{km}((b-b_{\TT_{j-1}^*})f \chi_{\TT_{j-1}^*})(\zeta)}\nonumber\\
&\qquad\qquad+\inf_{\zeta \in \TT_{j}}\ek{T_{km}((b-b_{\TT_{j}^*})f \chi_{\TT_{j}^*})(\zeta)}+\|b\|_{\text{BMO}}\inf_{\zeta \in \TT_{j}} \ek{T_{km}(f\chi_{\TT^*_{j}}(\zeta)}.\nonumber
\end{align}

Now to complete the proof we observe that $|\TT_{j}\setminus (F_{\TT_{j}, bad}\cup F_{\TT_{j-1}, bad})|>0$, indeed, $$|\TT_{j}\cap(F_{\TT_{j}, bad}\cup F_{\TT_{j-1}, bad})|\leq |F_{\TT_{j}, bad}|+ |\TT_{j}\cap F_{\TT_{j-1}, bad}|\leq \frac{1}{2}|\TT_j|+\frac{1}{4}|\TT_j|.$$ 

On account of the above, we obtain the following estimates 

$\displaystyle
\begin{array}{llcl}
\inf\limits_{\zeta \in \TT_{j}} &\mathcal{M}^{\sharp}_{T_{km}, c_{N}}(f \chi_{\TT_{j-1}^*})(\zeta) &\lesssim & c_{km} \ev{f}_{\TT^{*}_{j-1}},\\[1em] 
\inf\limits_{\zeta \in \TT_{j}}&|T_{km}(f \chi_{\TT_{j-1}^*})(\zeta)|&\lesssim & c_{km} \ev{f}_{\TT^{*}_{j-1}},\\[1em]
\inf\limits_{\zeta \in \TT_{j}}&|T_{km}(f \chi_{\TT_{j}^*})(\zeta)|&\lesssim & c_{km} \ev{f}_{\TT^{*}_{j}},\\[1em]
\inf\limits_{\zeta \in \TT_{j}}&|T_{km}((b-b_{\TT_{j-1}^*})f \chi_{\TT_{j-1}^*})(\zeta)|&\lesssim & c_{km} \ev{(b-b_{\TT_{j-1}^*})f}_{\TT_{j-1}^*},\\[1em]
\inf\limits_{\zeta \in \TT_{j}}& \mathcal{M}^{\sharp}_{T_{km}, c_{N}}((b-b_{\TT_{j-1}^*})f \chi_{\TT_{j-1}^*})(\zeta) &\lesssim & c_{km} \ev{(b-b_{\TT_{j-1}^*})f}_{\TT_{j-1}^*},\\[1em]
\inf\limits_{\zeta \in \TT_{j}}&|T_{km}((b-b_{\TT_{j-1}^*})f \chi_{\TT_{j-1}^*})(\zeta)|&\lesssim & c_{km} \ev{(b-b_{\TT_{j}^*})f}_{\TT_{j}^*}
\end{array}
$

Combining the above estimates with \cref{pointwise-recursive}, \cref{estIj} and \cref{estIIj}, we obtain
\begin{align*}
\big|[b, T_{km}](f \chi_{\PP^*})(z)\big|\lesssim_{N}\,  c_{km}\,& \left[\sum\limits_{\TT\in \GG}|b-b_{\TT^*}|\left(\fint_{\TT^{*}}|f|\right) \chi_{\TT}(z)+\sum\limits_{\TT\in \GG}\left(\fint_{\TT^{*}}|(b-b_{\TT^{*}})f|\right) \chi_{\TT}(z)\right]\\
&\qquad+ c_{km} \|b\|_{\text{BMO}} \sum\limits_{\TT\in \GG}\left(\frac{1}{|\TT^*|}\int_{\TT^{*}}|f|\right) \chi_{\TT}(z)    
\end{align*}
for a.e. $z\in \PP$. This completes the proof.
\end{proof}

\begin{theorem}
\label{Pointwise-estimate-sparse-local}
Let $\PP$ be any ultraparabolic cube. Then for any compactly supported bounded function $f$ there exists a $\frac{1}{2}$-sparse family $\GG\subset \mathscr{D}(\PP)$ such that
\begin{align}
\label{pointwise-sparse-local}
\big|T_{km}(f \chi_{\PP^*})(z)\big|\lesssim_{N}\,  c_{km}\, \sum\limits_{\TT\in \GG}\left(\frac{1}{|\TT|^*}\int_{\TT^*}|f|\right) \chi_{\TT}(z),    
\end{align}
holds for a.e. $z\in \PP$.
\end{theorem}
\begin{proof}
The proof is similar to that of \cref{Pointwise estimates-comm} and is inspired by the work \cite{Lorist-pointwaise-sparse2021}. We point out the main steps for sake of completion. 

\paragraph{\it\underline{Step 1}:} In a first step, we define an operator 
 \begin{align*}
\mathcal{N}f(z):=\max\{|T_{km}(f \chi_{\PP^*})(z)|, \mathcal{M}^{\sharp}_{T_{km}, s}(f \chi_{\PP^*})(z)\},
 \end{align*} which is weak-type $(1, 1)$ by \cref{weak-type} and \cref{weak-type-sharp}.
 
\paragraph{\it\underline{Step 2}:} In a second step, we consider the set 
\begin{align*}
F_{\PP,\,bad}:=\{z\in \PP: \max\{\mathcal{M}f(z), \mathcal{N}(f\chi_{\PP^*})(z)\}> c_{N} c_{km}  \langle f \rangle_{\PP^*} \},   
\end{align*} which satisfies 
\begin{align}
|F_{\PP,\,bad}|\lesssim O\left(\frac{1}{\kappa}\right)|\PP|\leq \frac{1}{2}|\PP|,    
\end{align}
for sufficiently large $c_N$. 

\paragraph{\it\underline{Step 3}:} Next, we apply \CZ~decomposition to the function $\chi_{F_{\PP,\,bad}}$ at level $\lambda>0$ which produces a  pairwise disjoint collection of ultraparabolic cubes $\GG_{\PP}\subset \UU(\PP)$ such that
\begin{align}
F_{\PP,\,bad}\subset \bigcup\limits_{\TT\in \GG_{\PP}}\TT\mbox{ and }\lambda<\frac{|\TT \cap F_{\PP, bad}|}{|\TT|}\leq \frac{1}{2},    
\end{align}
for all $\TT\in \GG_{\PP}$ for sufficiently small $\lambda$. We also get 
\begin{align}\sum\limits_{\TT\in \GG_{\PP}} |\TT|\leq \frac{1}{\lambda} \sum\limits_{\TT\in \GG_{\PP}} |\TT \cap F_{\PP,\,bad}|\lesssim \frac{1}{\lambda} |F_{\PP,\,bad}| \lesssim O\left(\frac{1}{c_N c_{km} \lambda }\right)|\PP| \leq \frac{1}{2} |\PP|,
\label{sparseness}
\end{align}
provided we choose $c_N$ very large.

\paragraph{\it\underline{Step 4}:} Now, the sparse collection is defined in a recursive manner. Set $\GG_{0}=\{\PP\}$ and $\GG_{1}=\GG_{\PP}$. For each $\TT\in \GG_{\PP}$, we repeat the steps 1, 2, and 3 to obtain the collection $\GG_{\TT}$ and their union $\GG_{2}=\bigcup_{\TT\in \GG_{\PP}} \GG_{\TT}$ gives us the second stage of the recursion. Iterating over all sub-cubes indefinitely we receive the collection $\GG_{k}$ for all $k=0, 1, 2, \ldots$ Once again, \cref{sparse-remark} and \cref{sparseness} imply that the collection $\GG=\bigcup\limits_{k\geq 0} \GG_{k}$ is a sparse collection.

\paragraph{\it\underline{Step 5}:} It remains to prove that $\GG$ satisfies the property \eqref{pointwise-sparse-local}. The recursive process implies that $\sum_{\TT\in \GG_{k}}|\TT|\leq \frac{1}{2^k}|\PP|$, consequently, $|\GG_k|\to 0$ as $k\to \infty$. Therefore, up to a measure zero set, for each point $z\in \PP$ there exists a largest number $k(z)$ such that $z\in \GG_{j}$ for each $0\leq j\leq k(z)$. Hence, there exist a chain of ultraparabolic dyadic cubes $\{\TT_{j}\}_{j=0}^{k(z)}$ with $\TT_{j}\in \GG_j$ such that $z\in \TT_{k(z)}\subset \TT_{k(z)-1}\subset\cdots\subset \TT_{0}=\PP$. Observe
\begin{align}
\nonumber\ek{T_{km}(f \chi_{\PP^*})(z)}&\leq \sum_{j=1}^{k(z)}\ek{T_{km}(f \chi_{\TT_{j-1}^*\setminus \TT_{j}^*})(z)}+\ek{T_{km}(f \chi_{\TT^{*}_{k(z)}})(z)}\\
&\lesssim \sum_{j=1}^{k(z)}\underbrace{\ek{T_{km}(f \chi_{\TT_{j-1}^*\setminus \TT_{j}^*})(z)}}_{I_j}+\ev{f}_{\TT^{*}_{k(z)}},
\label{pointwise-recursive2}
\end{align}
where in the last inequality we have used the fact that $z\notin \GG_{k(z)+1}$. 

Now, we shall estimate $I_j$ in the sum in \cref{pointwise-recursive2}. Note that for each $z'\in  \TT_{j}$, we can write 

Therefore,
\begin{align}\label{estIj2}
|T_{km}(f \chi_{\TT_{j-1}^*\setminus \TT_{j}^*})(z)|&\leq |T_{km}(f \chi_{\TT_{j-1}^*\setminus \TT_{j}^*})(z)-T_{km}(f \chi_{\TT_{j-1}^*\setminus \TT_{j}^*})(z')|\\
&\nonumber\qquad+|T_{km}(f \chi_{\TT_{j-1}^*\setminus \TT_{j}^*})(z')|\\
&\nonumber\leq \inf_{\zeta \in \TT_{j}} \mathcal{M}^{\sharp}_{T_{km}, c_{N}}(f \chi_{\TT_{j-1}^*})(\zeta)\\
&\nonumber\qquad+ \inf_{\zeta \in \TT_{j}}|T_{km}(f \chi_{\TT_{j-1}^*})(\zeta)|\\
&\nonumber\qquad\qquad+ \inf_{\zeta \in \TT_{j}}|T_{km}(f \chi_{\TT_{j}^*})(\zeta)|.
\end{align}

Now to complete the proof we observe that $|\TT_{j}\setminus (F_{\TT_{j},\,bad}\cup F_{\TT_{j-1},\,bad})|>0$, indeed, $$|\TT_{j}\cap(F_{\TT_{j},\,bad}\cup F_{\TT_{j-1},\, bad})|\leq |F_{\TT_{j},\,bad}|+ |\TT_{j}\cap F_{\TT_{j-1},\,bad}|\leq \frac{1}{2}|\TT_j|+\frac{1}{4}|\TT_j|.$$ 

Therefore

$
\begin{array}{lllc}
\inf\limits_{\zeta \in \TT_{j}}& \mathcal{M}^{\sharp}_{T_{km}, c_{N}}(f \chi_{\TT_{j-1}^*})(\zeta)&\lesssim & \kappa \ev{f}_{\TT^{*}_{j-1}},\\
\inf\limits_{\zeta \in \TT_{j}}&|T_{km}(f \chi_{\TT_{j-1}^*})(\zeta)|&\lesssim & \kappa \ev{f}_{\TT^{*}_{j-1}}\\
\inf\limits_{\zeta \in \TT_{j}}&|T_{km}(f \chi_{\TT_{j}^*})(\zeta)|&\lesssim & \kappa \ev{f}_{\TT^{*}_{j}}.
\end{array}$

Combining the above estimates in \cref{pointwise-recursive2} and \cref{estIj2}, we obtain
$$\big|T_{km}(f \chi_{\PP^*})(z)\big|\lesssim_{N}\,  c_{km}\, \sum\limits_{\TT\in \GG}\left(\frac{1}{|\TT^*|}\int_{\TT^{*}}|f|\right) \chi_{\TT}(z),$$
for a.e. $z\in \PP$. 
\end{proof}

\begin{proof}[Proof of \cref{represent_pointwise_bound}] Recall the representation formula
\begin{align*}
    u_{x_ix_j}=- \sum_{m=1}^\infty\sum_{k=1}^{g_m}c_{ij}^{km}T_{km}(\mathcal{L}u) + \sum_{h,k=1}^{s_0}  \sum_{m=1}^\infty\sum_{k=1}^{g_m} c_{ij}^{km}\lbrack a_{hk}, T_{km}\rbrack(u_{x_hx_k})+\alpha_{ij}\cdot\mathcal{L}u.
\end{align*}
Now the proof is complete a consequence of the applications of \cref{Pointwise-estimate-sparse-local} for $f=\mathcal{L}u$ and \cref{Pointwise estimates-comm} to $f=u_{x_hx_k}$. 
\end{proof}

\section{Applications of sparse domination}
In this section we highlight the consequences of our sparse domination results. 
\subsection{Weighted estimates on Lebesgue spaces}

This section is devoted to prove appropriate weighted norm inequalities for the sparse operators $\mathscr{A}_{\GG}$ and $\mathscr{A}_{\GG, b}$ where $b$ is in $(\delta, \kappa)-BMO$ class. Let us start with the definition of Muckenhoupt $A_p$ weights in our context.
\begin{definition}
\label{Muckenhoupt-ultraparabolic}
For $1<p <\infty$, we say a weight $\omega$ on $\mathbb{R}^{N+1}$ belongs to the class $\mathcal{A}_{p}(\mathbb{R}^{N+1}, d, |\cdot|)$, if 
\begin{eqnarray}
[\omega]_{\mathcal{A}_p}:= \sup_{\PP} \left( \frac{1}{|\PP|} \int_{\PP}\omega~ dx\right) \left( \frac{1}{|\PP|} \int_{\PP} \omega^{1-p^{\prime}} dx \right)^{p-1} <\infty, 
\end{eqnarray} 
where the supremum is taken over all ultraparabolic cubes $\PP$ in  the homogeneous space $(\mathbb{R}^{N+1}, d, |\cdot|)$. $\sigma$ will denote the dual weight to $\omega$, i.e., $\sigma=\omega^{-\frac{1}{p-1}}$. It is well known that $w\in \mathcal{A}_p(\mathbb{R}^{N+1}, d, |\cdot|) \iff \sigma\in \mathcal{A}_{p'}(\mathbb{R}^{N+1}, d, |\cdot|)$. Here on wards, we simply denote the class $\mathcal{A}_p(\mathbb{R}^{N+1}, d, |\cdot|)$ as $\mathcal{A}_p(\mathbb{R}^{N+1})$.
\end{definition}
We shall prove the following estimates in order to conclude weighted gradient estimates for the solution of \eqref{maineq}.
\begin{lemma}
\label{main-sparse-estimates}
Let $\GG$ be an $\eta$-sparse family. 
\begin{enumerate}
    \item 
For $1<p<\infty$, $\omega\in \mathcal{A}_p(\mathbb{R}^{N+1})$ and $f\in L^p(\omega)$ we have
\begin{equation}
\label{estmates-for-sparse-1}\|\mathscr{A}_{\GG}f\|_{L^p(\omega)}\leq C_{N, p, \eta} [\omega]_{\mathcal A_{p}}^{\max\left\{\frac{1}{p-1}, 1\right\}}\|f\|_{L^p(\omega)}.
\end{equation}
\item
Let $1<p<\infty$, $\omega\in \mathcal{A}_p(\mathbb{R}^{N+1})$, and $b\in L^{\infty}$ with small $(\delta, \mathfrak{K})$ norm where $\mathfrak{K}> r(\PP)$ for all $\PP\in \GG$. Then
\begin{equation}
\label{estmates-for-sparse-2}\|\mathscr{A}_{\GG}^{b}f\|_{L^p(\omega)}+\|\mathscr{A}_{\GG}^{b, *}f\|_{L^p(\omega)}\leq C_{N, p, \eta} [\omega]_{\mathcal A_{p}}^{2\max\left\{\frac{1}{p-1}, 1\right\}} O(\delta)\|f\|_{L^p(\omega)},
\end{equation}
holds for all $f\in L^p(\omega)$.
\end{enumerate}
\end{lemma}
\begin{proof}
The proof of \eqref{estmates-for-sparse-1} is well known and is available in \cite{Lorist-pointwaise-sparse2021}. We only prove \eqref{estmates-for-sparse-2} by showing that $\mathscr{A}_{\GG}^{b, *}$ is controlled pointwise by composition of two sparse operators. The argument is similar to the proof of Theorem~\ref{Pointwise-estimate-sparse-local}. Let us divide the proof into several parts. 
\paragraph{\it\underline{Step 1} (Pointwise controlling the oscillation):} We claim that for any $\PP\in \GG$ there exists a family of pairwise disjoint cubes $\GG_{\PP}\subset \UU(\PP)$ such that
\begin{align}
\label{Claim-1}
|b(z)-(b)_{\PP}|\leq \alpha_{N}\,  \Omega_{\PP}(b)+\sum_{\TT\in \GG_{\PP}}|b(z)-(b)_{\TT}| \chi_{\TT},\, \, \text{a.e.}\,\, x\in \PP,
\end{align}
and $\sum_{\TT\in \GG_{\PP}}|\TT|\leq \frac{1}{2}|\PP|$. Let us consider the set $E:=\{x\in \PP: \displaystyle\mathcal{M}^{\PP}(b-(b)_{\PP})(z)> \alpha_{N}\,  \Omega_{\PP}(b)\}$. Choose $\alpha_{N}$ large enough to ensure that $|E|\leq \frac{1}{2}|\PP|$. Local \CZ~decomposition to the function $\chi_{E}$ at level $\lambda>0$ produces a  pairwise disjoint collection of ultraparabolic cubes $\GG_{\PP}\subset \UU(\PP)$ such that
\begin{align}
\label{LCZ}
E\subset \bigcup\limits_{\TT\in \GG_{\PP}}\TT\mbox{ and }\lambda<\frac{|\TT \cap E|}{|\TT|}\leq \frac{1}{2},    
\end{align}
for all $\TT\in \GG_{\PP}$ for sufficiently small $\lambda$. Therefore, $|b(z)-(b)_{\PP}|\leq \alpha_{N} \Omega_{\PP}(b)$ a.e. $z\notin \bigcup\limits_{\TT\in \GG_{\PP}}\TT$. Also, \eqref{LCZ} implies that $\TT\cap E^{c}\neq \emptyset$, hence for some $\xi\in \TT\cap E^{c}$ we have  \begin{align*}
    |(b)_{\TT}-(b)_{\PP}|\leq \frac{1}{|\TT|}\int_{\TT}|b-(b)_{\PP}|\leq \mathcal{M}^{\PP}(b-(b)_{\PP})(\xi)\leq \alpha_{N} \Omega_{\PP}(b).
\end{align*}
Combining the above estimates we obtain 
\begin{align*}
|b(z)-(b)_{\PP}|&\leq |b(z)-(b)_{\PP}|\chi_{\PP\setminus \bigcup\limits_{\TT\in \GG_{\PP}}\TT}+\sum_{\TT\in \GG_{\PP}}|b(z)-(b)_{\TT}| \chi_{\TT}\\
&+\sum_{\TT\in \GG_{\PP}}|(b)_{\PP}-(b)_{\TT}| \chi_{\TT}\\
&\leq \alpha_{N}\,  \Omega_{\PP}(b)+\sum_{\TT\in \GG_{\PP}}|b(z)-(b)_{\TT}| \chi_{\TT}.
\end{align*}
Also, arguing as in Theorem~\ref{Pointwise-estimate-sparse-local}, we obtain $\sum_{\GG_{\PP}}|\TT|\leq \frac{1}{2}|\PP|$. This completes the claim. Recursive application of \eqref{Claim-1} produces a $\frac{1}{2}-$sparse family $\mathcal{Z}(\PP)\subset \UU(\PP)$ such that \begin{align}
\label{Recursive-oscillation}
|b(z)-(b)_{\PP}|\leq \alpha_{N}\,  \sum_{\TT\in \mathcal{Z}(\PP)}\Omega_{\TT}(b)\chi_{\TT}\, \, \text{a.e.}\,\, z\in \PP.
\end{align}

\paragraph{\it\underline{Step 2} (Augmentation process):} Recursive application to each $\PP\in \GG$ produces $\frac{1}{2}-$sparse families $\mathcal{Z}(\PP)\subset \UU(\PP)$ such that \eqref{Recursive-oscillation} holds. Define the new collection $\tilde{\GG}=\GG\cup (\bigcup_{\PP\in \GG}\mathcal{Z}(\PP))$. Since $\GG$ is a $\eta$-sparse family and each $\mathcal{Z}(\PP)$ is $\frac{1}{2}-$sparse, using Lemma~2.4 from \cite{Lerner-comm}, we conclude that $\tilde{\GG}$ is $c_{\eta}:=\frac{\eta}{2(1+\eta)}$ sparse. Also, for any arbitrary cube $\PP\in \tilde{\GG}$, denote $S(\PP):=\{\TT\in \tilde{\GG}: \TT\subseteq  \PP\}$. Let us decompose $S(\PP)$ into maximal subfamilies and using the fact that $\tilde{\GG}$ is $c_{\eta}$ sparse, we conclude that
\begin{align}
\label{main-oscillation-estimate}
\nonumber|b(z)-(b)_{\PP}|&\leq \alpha_{N} \,  \sum_{\TT\in S(\PP)}\Omega_{\TT}(b)\chi_{\TT}(z)\\
&\leq \alpha_{N}\, O(\delta) \sum_{\TT\in S(\PP)}\chi_{\TT}(z), \,\,\text{a.e.}\,\, z\in \PP,
\end{align}
where in the last inequality we have used the fact $\Omega_{\TT}(b)\leq O(\delta)$ since $b$ has small $(\delta, \mathfrak{K})-$ BMO norm. 

\paragraph{\it\underline{Step 3} (Estimate of $\|\mathscr{A}_{\GG}^{b, *}f\|_{L^p(\omega)}$):} 
The sparseness of $\tilde{\GG}$ implies there exists ${F}_{\TT}\subset \TT$ such that $|{F}_{\TT}|\geq c_{\eta} |\TT|$ and $\{{F}_{\TT}\}$ is a disjoint family. Now invoking \eqref{main-oscillation-estimate}, for each ultraparabolic cube $\PP\in \tilde{\GG}$, we obtain
\begin{align}
\nonumber\int_{\PP} |b(z)-(b)_{\PP}||f(z)|&\leq \alpha_{N} O(\delta) \sum_{\TT\in S(\PP)} \left(\frac{1}{|\TT|}\int_{\TT} |f|\right) |\TT| \\
&\leq c_{\eta}\, \alpha_{N}\, O(\delta) \sum_{\TT\in S(\PP)} \left(\frac{1}{|\TT|}\int_{\TT} |f|\right) |F_{\TT}|\leq c_{\eta, N}\, O(\delta) \int_{\PP} \mathscr{A}_{S(\PP)}(|f|)(z)\, dz.
\end{align}
Therefore, 
\begin{align}
\label{control-on-sparse-comm}
\mathscr{A}_{\GG}^{b, *}f(z)\leq c_{\eta, N}\, O(\delta) \mathscr{A}_{\GG}(\mathscr{A}_{\tilde{\GG}}|f|)(z).
\end{align}
From \eqref{estmates-for-sparse-1}, we conclude that $$\|\mathscr{A}_{\GG}^{b, *}f\|_{L^p(\omega)}\leq c_{\eta, N} O(\delta) \|\mathscr{A}_{\GG}(\mathscr{A}_{\tilde{\GG}}|f|)\|_{L^p(\omega)}\leq c_{\eta, N} O(\delta)  [\omega]_{\mathcal A_{p}}^{2\max\left\{\frac{1}{p-1}, 1\right\}} \|f\|_{L^p(\omega)}.$$
Finally $\|\mathscr{A}_{\GG}^{b}\|_{L^p(\omega)\to L^p(\omega)}=\|\mathscr{A}_{\GG}^{b, *}f\|_{L^{p'}(\sigma)\to L^{p'}(\sigma)}\leq c_{\eta, N} O(\delta)  [\omega]_{\mathcal A_{p}}^{2\max\left\{\frac{1}{p-1}, 1\right\}}.$
\end{proof}
Now, in order to prove Theorem~\ref{sharp-weighted-estimates}, we begin with the following local estimate.

\begin{proposition}\label{estonball1}
For every $p\in (1,\infty)$, there exists $C=C(\text{data})$ and $r_0=r_0(\text{data})$ such that if $u\in C_0^\infty(\RR^N\times(0,\infty))$, $\text{spt}\,u\subset B_r\subset \RR^N\times(0,\infty)$ with $0<r<r_0$, then, for $i,j=1,2,\ldots,s_0$
\begin{align*}
    \|u_{x_ix_j}\|_{L^p(\omega)}\leq C_{A, B, N, p, \eta}\, [\omega]_{\mathcal A_{p}}^{\max\left\{\frac{1}{p-1}, 1\right\}}\|\mathcal{L}u\|_{L^p(\omega)}.
\end{align*}
\end{proposition}

\begin{proof}
There is an ultraparabolic cube $\PP$ such that $B_r\subset \PP$ whose size is determined in \cref{main-sparse-estimates} and this determines $r_0$. Therefore, we take $L^p(\PP,\omega)$ norms on both sides of the representation formula \cref{represent3-sparse} to get
\begin{align*}
    \|u_{x_ix_j}\|_{L^{p}(\PP,\omega)} &\leq \kappa_{A, N, C} \sum_{m=1}^\infty\sum_{k=1}^{g_m}c_{km} |c_{ij}^{km}| \sum_{l=1}^{C(N)}\,  \|\mathscr{A}_{\GG_{km, l}}(\mathcal{L}u)\|_{L^{p}(\PP,\omega)} \nonumber\\
    &\quad+  \sum_{h,k=1}^{s_0}  \sum_{m=1}^\infty\sum_{k=1}^{g_m}c_{km} |c_{ij}^{km}| \sum_{l=1}^{C(N)} \|\mathscr{A}_{\GG_{km, l} }^{a_{hk}}(u_{x_hx_k})\|_{L^{p}(\PP,\omega)}\nonumber\\
    &\quad\quad+\sum_{h,k=1}^{s_0}  \sum_{m=1}^\infty\sum_{k=1}^{g_m}c_{km} |c_{ij}^{km}| \sum_{l=1}^{C(N)} \, \|\mathscr{A}_{\GG_{km, l} }^{a_{hk}, *}(u_{x_hx_k})\|_{L^{p}(\PP,\omega)}+|\alpha_{ij}|\cdot\|\mathcal{L}u\|_{L^{p}(\PP,\omega)}.
\end{align*}
Applying \cref{main-sparse-estimates} to the above inequality, we obtain
\begin{align}\label{gradest1}
    \|u_{x_ix_j}\|_{L^{p}(\PP,\omega)} &\leq C_{A, B, N, p, \eta}\, [\omega]_{\mathcal A_{p}}^{\max\left\{\frac{1}{p-1}, 1\right\}}\,\|\mathcal{L}u\|_{L^{p}(\PP,\omega)} \sum_{m=1}^\infty\sum_{k=1}^{g_m}c_{km} |c_{ij}^{km}|  \nonumber\\
    &\quad+  C_{A,B, N, p, \eta} [\omega]_{\mathcal A_{p}}^{2\max\left\{\frac{1}{p-1}, 1\right\}} O(\delta) \sup\limits_{h,k}\|u_{x_hx_k}\|_{L^{p}(\PP,\omega)}  \sum_{m=1}^\infty\sum_{k=1}^{g_m}c_{km} |c_{ij}^{km}| \nonumber\\
    &\quad\quad+|\alpha_{ij}|\cdot\|\mathcal{L}u\|_{L^{p}(\PP,\omega)}.
\end{align}

We estimate the sum
\begin{align*}
    \sum_{m=1}^\infty\sum_{k=1}^{g_m}c_{km} |c_{ij}^{km}|&\leq \sum_{m=1}^\infty\sum_{k=1}^{g_m} m^{-2l}m^{\frac{N+1}{2}}\\
    &\leq \sum_{m=1}^\infty m^{-2l}m^{\frac{3N-1}{2}}<\infty
\end{align*} provided $l>\frac{3N-1}{4}$. In the first inequality, we have used \cref{riemlebdecay}, \cref{kernelprop3} and \cref{horpropthm}, whereas for the second inequality we used \cref{dimSphHar}. Therefore \cref{gradest1} if further estimated to
\begin{align}\label{gradest2}
    \|u_{x_ix_j}\|_{L^{p}(\PP,\omega)} &\leq C_{A, B, N, p, \eta}\, \max\{[\omega]_{\mathcal A_{p}}^{\max\left\{\frac{1}{p-1}, 1\right\}},|\alpha_{ij}|\}\,\|\mathcal{L}u\|_{L^{p}(\PP,\omega)} \nonumber\\
    &\quad+  C_{A,B, N, p, \eta} [\omega]_{\mathcal A_{p}}^{2\max\left\{\frac{1}{p-1}, 1\right\}} O(\delta) \sup\limits_{h,k}\|u_{x_hx_k}\|_{L^{p}(\PP,\omega)},
\end{align}
for all $i,j=1,2,\ldots,s_0$. Therefore, if $\delta$ is further chosen such that 
\begin{align*}
    C_{A,B, N, p, \eta} [\omega]_{\mathcal A_{p}}^{2\max\left\{\frac{1}{p-1}, 1\right\}} O(\delta) < 1,
\end{align*}
we get 
\begin{align}\label{gradest3}
    \sup\limits_{i,j}\|u_{x_ix_j}\|_{L^{p}(\PP,\omega)} &\leq C_{A, B, N, p, \eta}\, \max\{[\omega]_{\mathcal A_{p}}^{\max\left\{\frac{1}{p-1}, 1\right\}},|\alpha_{ij}|\}\,\|\mathcal{L}u\|_{L^{p}(\PP,\omega)}\\
    &\nonumber\leq C_{A, B, N, p, \eta}\, [\omega]_{\mathcal A_{p}}^{\max\left\{\frac{1}{p-1}, 1\right\}}\,\|\mathcal{L}u\|_{L^{p}(\PP,\omega)},
\end{align}
since $|\alpha_{ij}|\lesssim 1$.
Now, since $u$ is supported in $B_r$, we obtain the desired estimate.
\end{proof}

Now, the proof of \cref{sharp-weighted-estimates} is immediate.

\begin{proof}[Proof of Theorem~\ref{sharp-weighted-estimates}] In a first step, one obtains the following estimate by a standard argument using cut-off functions and an interpolation lemma, for example, see \cite[Theorem 9.11]{GT2001}.

\begin{align}\label{gradest4}
    \sup\limits_{i,j}\|u_{x_ix_j}\|_{L^{p}(B_r,\omega)} &\leq C_{A, B, N, p, \eta}\, [\omega]_{\mathcal A_{p}}^{\max\left\{\frac{1}{p-1}, 1\right\}}\,\left(\|\mathcal{L}u\|_{L^{p}(B_{2r},\omega)}+\frac{\|u\|_{L^{p}(B_{2r},\omega)}}{r^2}\right).
\end{align}

The estimate in \cref{sharp-weighted-estimates} follows by covering $\Om'$ with balls $B$ of radius $R/2$ where where $R$ has a further restriction $R<\text{dist}(\Om',\Om)$, other than the restriction from \cref{estonball1}

\end{proof}

\subsection{Weighted estimates on Variable Lebesgue spaces}\label{sparsedomVarLeb}


Let us begin by introducing the variable Lebesgue space and their Sobolev counterparts. We denote by $p(\cdot)=p(z):\RR^{N+1}\to \RR$ a measurable function satisfying
\begin{align}\label{boundsonp}
    1<\gamma_1\leq p(z)\leq \gamma_2 < \infty,\,\mbox{ for all }z\in\RR^{N+1}
\end{align} and we call it a {\it{variable exponent}}. The {\it{variable exponent Lebesgue space}} $L^{p(\cdot)}(\Om)$ is defined to be the space of all functions $g:\Om\to\RR$ such that the {\it{modular} }
\begin{align*}
    \rho_{p(\cdot)}(g):=\int_{\Om}|g(z)|^{p(z)}\,dx
\end{align*} is finite. 

The {\it{variable exponent Lebesgue space}} $L^{p(\cdot)}(\Om)$ is known to be a reflexive Banach space when equipped with the Luxembourg norm
\begin{align}
    \|g\|_{L^{p(\cdot)}(\Om)}:=\inf\left\{\lambda>0:\rho_{p(\cdot)}\left(\frac{g}{\lambda}\right)\leq 1\right\},
\end{align} and the hypothesis \cref{boundsonp}. For a proof, see \cite{Diening2017}. With the hypothesis \cref{boundsonp}, the dual space $L^{p'(\cdot)}(\Om)$ is also reflexive where $p'(z)=\frac{p(z)}{p(z)-1}$.

The {\it{variable exponent Sobolev space}} $W^{k,p(\cdot)}(\Om)$ is defined as the space of functions $g\in L^{p(\cdot)}(\Om)$ such that the distributional derivatives $D^{\beta}g$ is also in $L^{p(\cdot)}(\Om)$ for all $|\beta|\leq k$ where $\beta$ is a multi-index. Further, the norm in $W^{k,p(\cdot)}(\Om)$ is defined to be 
\begin{align}
    \|g\|_{W^{k,p(\cdot)}(\Om)}:=\sum\limits_{|\beta|\leq k}\|D^{\beta} g\|_{L^{p(\cdot)}(\Om)}.
\end{align}

The class of variable exponents under consideration is very large to obtain regularity properties of the partial differential equations under consideration. Zhikov \cite{Zhikov86}, while studying homogenization of highly anisotropic nonlinear media, was among the first people to notice that the variable exponent needs to satisfy additional properties in order to avoid the so called {\it Lavrentiev phenomena}, which may be paraphrased as the principle that equations (and their structure/symmetry properties) determine the function spaces in which they should be solved. Zhikov zeroed in on log-H\"older continuity of the variable exponent as an essential requirement in order to obtain regularity results for minimizers of functionals with $p(\cdot)$-growth. Diening\cite{Diening2004} proved that the log-H\"older condition is also required to prove boundedness of the maximal operator on variable exponent Lebesgue spaces. This clearly suggests that the log-H\"older condition is also indispensable when studying gradient $L^p(\cdot)$ estimates.

\begin{definition}
We say that a variable exponent $p(\cdot)$ is locally log-H\"older continuous if there exists a constant $C_0$ such that for every $\eta,\zeta\in\RR^{N+1}$ such that $d(\eta,\zeta)<1/2$, it holds that
\begin{align}\label{localloghold}
    |p(\eta)-p(\zeta)|\leq \frac{C_0}{-\log(d(\eta,\zeta))}.
\end{align}
On the other hand, $p(\cdot)$ is log-H\"older continuous at infinity with respect to some base point $\zeta_0\in \RR^{N+1}$, if there are constants $C_\infty$ and $p_\infty$ such that for every $\zeta\in \RR^{N+1}$,
\begin{align}\label{infiloghold}
    |p(\zeta)-p_\infty|<\frac{C_\infty}{\log(e+d(\zeta,\zeta_0))}.
\end{align}
If the variable exponent $p(\cdot)$ satisfies \cref{localloghold} and \cref{infiloghold}, we say that $p(\cdot)$ is a globally log-H\"older continuous function.
\end{definition}

\begin{remark}
It turns out that the definition of log-H\"older at infinity is indpendent of the choice of base point. Hence, there is no loss of generality to choose an arbitrary $\zeta_0$. In our applications, we will only require the boundedness of the maximal function for $L^{p(\cdot)}(\Om)$ for bounded domains $\Om$. In this case, the log-H\"older condition at infinity is superfluous.
\end{remark}

Let us start with the definition of Muckenhoupt $A_p$ weights in our context.
\begin{definition}
\label{Muckenhoupt-ultraparabolic-variable}
For $1<p <\infty$, we say a weight $\omega$ on $\mathbb{R}^{N+1}$ belongs to the class $\mathcal{A}_{p(\cdot)}(\mathbb{R}^{N+1}, d, |\cdot|)$, if 
\begin{eqnarray}
\label{Variable-Muckenhoupt}
[\omega]_{\mathcal{A}_{p(\cdot)}}:= \sup_{\PP}  |\PP|^{-1} \|\omega \chi_{\PP}\|_{L^{p(\cdot)}}\|\omega^{-1}\chi_{\PP}\|_{L^{p'(\cdot)}}<\infty.
\end{eqnarray} 
where the supremum is taken over all ultraparabolic cubes $\PP$ in  the homogeneous space $(\mathbb{R}^{N+1}, d, |\cdot|)$. Here onwards, we simply denote the class $\mathcal{A}_{p(\cdot)}(\mathbb{R}^{N+1}, d, |\cdot|)$ as $\mathcal{A}_{p(\cdot)}(\mathbb{R}^{N+1})$.
\end{definition}

We shall require the following result on the boundedness of the maximal function on weighted variable exponent Lebesgue spaces for homogeneous spaces which we state for the ultraparabolic Lebesgue space $L^{p(\cdot)}(\mathbb{R}^{N+1}, d, |\cdot|)$.

\begin{proposition}(\cite[Theorem 1.13]{Uribe2022})\label{maxboundvar}
Suppose that the variable exponent $p(\cdot)$ is globally log-H\"older continuous and satisfies \cref{boundsonp} then
\begin{align}
    \|(\mathcal{M} f)\|_{L^{p(\cdot)}(\omega)}\leq C\|f\|_{L^{p(\cdot)}(\omega)}
\end{align} if and only if the weight $w\in \mathcal{A}_{p(\cdot)}(\mathbb{R}^{N+1})$,
\end{proposition} where
\begin{align*}
    \|f\|_{L^{p(\cdot)}(\omega)}:=\|f\omega\|_{L^{p(\cdot)}(\RR^{N+1})}
\end{align*}

\begin{lemma}
\label{sparse-estimates-variable}
Let $\GG$ be an $\eta$-sparse family. 
\begin{enumerate}
    \item 
For $1<p<\infty$, $\omega\in \mathcal{A}_{p(\cdot)}(\mathbb{R}^{N+1})$ and $f\in L^{p(\cdot)}(\omega)$ we have
\begin{equation}
\label{estmates-for-sparse-1-variable}\|\mathscr{A}_{\GG}f\|_{L^{p(\cdot)}(\omega)}\leq C_{N, p, \eta, \omega}\|f\|_{L^{p(\cdot)}(\omega)}.
\end{equation}
\item
Let $1<p<\infty$, $\omega\in \mathcal{A}_{p(\cdot)}(\mathbb{R}^{N+1})$, and $b\in L^{\infty}$ with small $(\delta, \mathfrak{K})$ norm where $\mathfrak{K}> r(\PP)$ for all $\PP\in \GG$. Then
\begin{equation}
\label{estmates-for-sparse-2-variable}\|\mathscr{A}_{\GG}^{b}f\|_{L^{p(\cdot)}(\omega)}+\|\mathscr{A}_{\GG}^{b, *}f\|_{L^{p(\cdot)}(\omega)}\leq C_{N, p, \eta, \omega} O(\delta)\|f\|_{L^{p(\cdot)}(\omega)}
\end{equation}
holds for all $f\in L^p(\omega)$.
\end{enumerate}
\end{lemma}
\begin{proof}
Since the proof is very similar to the proof of Lemma~\ref{main-sparse-estimates}, we give a brief sketch of the proof.

\paragraph{\it\underline{Proof of \cref{estmates-for-sparse-1-variable}} :}
The proof relies on duality. Let $g\in L^{p'(\cdot)}(\omega^{-1})$ be a non-negative function. Then
\begin{align*}
\int_{\mathbb{R}^{N+1}} \mathscr{A}_{\GG}f(z) g(z)\, dz &\leq \sum_{\TT\in \GG} \left(\frac{1}{|\TT|}\int_{\TT}|f| \right) \left(\frac{1}{|\TT|}\int_{\TT} g \right) |\TT| \\
&\lesssim_{\eta} \sum_{\TT\in \GG} \left(\frac{1}{|\TT|}\int_{\TT}|f| \right) \left(\frac{1}{|\TT|}\int_{\TT} g \right) |F_\TT|\\
&\lesssim_{\eta} \sum_{\TT\in \GG} \inf_{z\in \PP} \mathcal{M}f(z)\inf_{z\in \PP} \mathcal{M}g(z) |F_\TT|\\
&\lesssim_{\eta} \sum_{\TT\in \GG}  \int_{F_{\TT}}\mathcal{M}f(z) \mathcal{M}g(z)\, dz\\
&\lesssim_{\eta} \int_{\mathbb{R}^{N+1}}\mathcal{M}f(z) \mathcal{M}g(z)\, dz\lesssim_{\eta} \|\mathcal{M}f\|_{L^{p(\cdot)}(\omega)} \|\mathcal{M}g\|_{L^{p'(\cdot)}(\omega^{-1})}\\
&\stackrel{\redlabel{a}{a}}{\lesssim}_{\eta, \omega} \|f\|_{L^{p(\cdot)}(\omega)} \|g\|_{L^{p'(\cdot)}(\omega^{-1})},
\end{align*} where \redref{a}{a} follows from \cref{maxboundvar}.
Now the proof is finished using duality.

\paragraph{\it\underline{Proof of \cref{estmates-for-sparse-2-variable}} :}
In \cref{main-sparse-estimates}, we have proved the following pointwise estimate
\begin{align*}
\mathscr{A}_{\GG}^{b, *}f(z)\leq c_{\eta, N}\, O(\delta) \mathscr{A}_{\GG}(\mathscr{A}_{\tilde{\GG}}|f|)(z),
\end{align*}
where $\tilde{\GG}$ is a $c_{\eta}$-sparse family containing $\GG$. Now \cref{estmates-for-sparse-1-variable} implies the following
\begin{align*}
\|\mathscr{A}_{\GG}^{b, *}f\|_{L^{p(\cdot)}(\omega)}&\leq c_{\eta, N}\, O(\delta) \|\mathscr{A}_{\GG}(\mathscr{A}_{\tilde{\GG}}|f|)\|_{L^{p(\cdot)}(\omega)}\\
&\leq c_{\eta, N, \omega}\, O(\delta) \|\mathscr{A}_{\tilde{\GG}}(f)\|_{L^{p(\cdot)}(\omega)}\\
&\leq c_{\eta, N, \omega}\, O(\delta) \|f\|_{L^{p(\cdot)}(\omega)}.  
\end{align*}
Thus the proof is complete.
\end{proof}

Now, in order to prove Theorem~\ref{main-theorem2}, we begin with the following local estimate.

\begin{proposition}\label{estonball2}
There exists $C=C(\text{data})$ and $r_0=r_0(\text{data})$ such that if $u\in C_0^\infty(\RR^N\times(0,\infty))$, $\text{spt}\,u\subset B_r\subset \RR^N\times(0,\infty)$ with $0<r<r_0$, then, for $i,j=1,2,\ldots,s_0$
\begin{align*}
    \|u_{x_ix_j}\|_{L^{p(\cdot)}(\omega)}\leq C_{A, B, N, p, \omega,\eta}\, \|\mathcal{L}u\|_{L^{p(\cdot)}(\omega)},
\end{align*} where the variable exponent $p(\cdot)$ satisfies the log-H\"older condition and $\omega\in\mathcal{A}_{p(\cdot)}$.
\end{proposition}

\begin{proof}
The proof is exactly similar to the proof of \cref{estonball1}. The only difference is we apply \cref{sparse-estimates-variable} to the pointwise representation in \cref{represent_pointwise_bound}.
\end{proof}

Now, the proof of \cref{main-theorem2} is immediate.

\begin{proof}[Proof of Theorem~\ref{main-theorem2}] In a first step, one obtains the following estimate by a standard argument using cut-off functions and an interpolation lemma, for example, see \cite[Theorem 9.11]{GT2001}.

\begin{align}\label{gradest4.5}
    \sup\limits_{i,j}\|u_{x_ix_j}\|_{L^{p(\cdot)}(B_r,\omega)} &\leq C_{A, B, N, p, \eta,\omega}\, \left(\|\mathcal{L}u\|_{L^{p(\cdot)}(B_{2r},\omega)}+\frac{\|u\|_{L^{p(\cdot)}(B_{2r},\omega)}}{r^2}\right).
\end{align}

Now, using the fact that
\begin{align*}
    Yu=\mathcal{L}u-\sum_{i,j=1}^{s_0}a_{ij}(z)u_{x_ix_j}
\end{align*}
we readily obtain through \cref{gradest4.5} that

\begin{align}\label{gradest5}
    \|Yu\|_{L^{p(\cdot)}(B_r,\omega)} &\leq C_{A, B, N, p, \eta,\omega}\, \left(\|\mathcal{L}u\|_{L^{p(\cdot)}(B_{2r},\omega)}+\frac{\|u\|_{L^{p(\cdot)}(B_{2r},\omega)}}{r^2}\right).
\end{align}

Now we deduce the estimate in \cref{main-theorem2} by covering $\Om'$ with balls $B$ of radius $R/2$ where where $R$ has a further restriction $R<\text{dist}(\Om',\Om)$, other than the restriction from \cref{estonball1}

\end{proof}

\begin{remark}
In the next section, we will study estimates in generalized Orlicz spaces which seemingly subsume the foregoing Lebesgue and variable exponent Lebesgue spaces. However, a full-fledged theory of weights is not yet developed for the generalized Orlicz spaces and in fact, the assumptions required for the boundedness of the maximal functions on generalized Orlicz spaces do not allow for the introduction of sharp weights when restricted to specialized classes such as the variable exponent Lebesgue spaces. See \cite{UribeHasto2018} for a related discussion. For this reason, we treat the case of {\it variable exponent Lebesgue spaces} separately.
\end{remark}

\subsection{Estimates on generalized Orlicz spaces}\label{SparseDomGenOrc}

We state the basic facts on generalized Orlicz spaces. For details, we refer to the books \cite{Diening2017} and \cite{Harjulehto2019}. The concept of Orlicz spaces has found a great synthesis in the work \cite{Harjulehto2019}. An important success of such spaces is reflected in the great generality in which a number of problems from regularity theory can be studied in a unified manner while often recovering optimal results in specialized cases. An instance of this may be seen in the following works \cite{Hasto-JDE-elliptic,ByunOh2020,hasto_maximal_2021}.

A convex, left-continuous function $\varphi:[0,\infty)\to[0,\infty)$ satisfying $\varphi(0)=\lim\limits_{s\to 0^+}\varphi(s)=0$ and $\lim\limits_{s\to \infty}\varphi(s)=\infty$ is called a $\Phi$-function. We denote the set of all $\Phi$-functions by $\Phi$.
Given a subset $E\subset\RR^{N+1}$, we define $\Phi(E)$ to be the set of all functions $\varphi:E\times [0,\infty)\to[0,\infty)$ such that $\varphi(z,\cdot)\in \Phi$ for each $z\in E$ and $\varphi(\cdot,s)$ is a measurable function for each $s\geq 0$.

 The {\it{generalized Orlicz space}} or the Musielak-Orlicz space $L^{\varphi(\cdot)}(E)$ is defined to be the space of all functions $g:E\to\RR$ such that the {\it{modular} }
\begin{align*}
    \rho_{\varphi(\cdot)}(g):=\int_{E}\varphi(z,|g(z)|)\,dz
\end{align*} is finite. 

The {\it{generalized Orlicz space}} $L^{\varphi(\cdot)}(E)$ can be equipped with the Luxembourg norm
\begin{align}
    \|g\|_{L^{\varphi(\cdot)}(E)}:=\inf\left\{\lambda>0:\rho_{\varphi(\cdot)}\left(\frac{g}{\lambda}\right)\leq 1\right\}.
\end{align}

One can impose an equivalence relation on the class $\Phi(E)$ where $\varphi\in\Phi(E)$ is equivalent to $\psi\in\Phi(E)$ if there is $L>1$ such that $\varphi(z,t/L)\lesssim \psi(z,t)\lesssim \varphi(z,Lt)$ for all $z\in E$ and $t\geq 0$. Equivalent $\Phi$ functions give rise to the same generalized Orlicz spaces with comparable norm. Hence, in many cases, it may be sufficient to consider $\varphi\in\Phi(E)$ with better properties. For example, every $\Phi(E)$ function has a convex (in $t$) equivalent.

A function $\varphi^{-1}:[0,\infty]\to[0,\infty]$ is called the left inverse of $\varphi$ and defined by $\varphi^{-1}(s)=\inf\{t:\varphi(t)\geq s\}$.

The function $\varphi^{*}:[0,\infty)\to[0,\infty]$ defined by $\varphi^{*}(s)=\sup\limits_{t\geq 0}\{ts-\varphi(t)\}$ is called the conjugate of $\varphi$.

The following H\"older's inequality holds for generalized Orlicz spaces.

\begin{lemma}(\cite[Lemma 3.2.11]{Harjulehto2019})\label{holdergenOrc}
Let $\varphi\in \Phi(E)$, then
\begin{align*}
    \int\limits_E f\,g\,dz \leq 2 \|f\|_{L^{\varphi(\cdot)}(E)}\,\|g\|_{L^{\varphi^*(\cdot)}(E)}
\end{align*} for all $f\in L^{\varphi(\cdot)}(E)$ and $g\in {L^{\varphi^*(\cdot)}(E)}$.
\end{lemma}

Let $\varphi\in\Phi(E)$. For $U\subset E$, we use the notation
\begin{align*}
    \varphi_U^+(s)=\sup\limits_{z\in U}\varphi(z,s),\qquad \varphi_U^-(s)=\inf\limits_{z\in U}\varphi(z,s).
\end{align*}
The notations $\varphi_E^+,\,\varphi_E^-$ will be shortened to $\varphi^+,\,\varphi^-$ respectively.

A function $g:(0,\infty)\to\RR$ is said to be $L$-almost increasing (decreasing) with $L\geq (\leq)1$, if $g(s)\leq(\geq) L\,g(t)$ for all $0<s<t$. 

Let $E$ be a domain on $\RR^{N+1}$. We restrict the class of $\Phi(E)$ functions to those satisfying the following hypothesis.

\begin{description}
\descitem{A0}{A0} There exists $\alpha\in (0,1)$ such that $\varphi(z,\alpha)\leq 1\leq \varphi(z,1/\alpha)$ for every $z\in\RR^{N+1}$.
\descitem{A1}{A1} There exists $\beta\in (0,1)$ such that 
\begin{align*}
    \beta\varphi^{-1}(z,s)\leq \varphi^{-1}(z',s)
\end{align*} for every $s\in\left[1,\frac{1}{|B|}\right]$ almost everywhere $z,z'\in B\cap E$
 and every ball $B$ such that $|B|\leq 1$.
\descitem{A2}{A2} For every $s>0$, there exists $\beta\in(0,1]$ and $h\in L^1(E)\cap L^\infty(E)$ such that
\begin{align*}
    \varphi(\zeta,\beta t)\leq \varphi(\eta,t)+h(\zeta)+h(\eta).
\end{align*} for almost all $\zeta,\eta\in E$ and $\varphi(y,t)\in [0,s]$.
\descitem{aInc$_p$}{aInc$_p$} There exists $p>1$ and $L\geq 1$ such that $\frac{\varphi(z,s)}{s^p}$ is $L$-almost increasing with respect to $s$ for every $z\in E$.
\descitem{aDec$_q$}{aDec$_q$} There exists $p>1$ and $L\leq 1$ such that $\frac{\varphi(z,s)}{s^q}$ is $L$-almost decreasing with respect to $s$ for every $z\in E$.
\end{description}

\begin{remark}
If $E$ is bounded, \descref{A2}{A2} is superfluous. See \cite[Lemma 4.2.3]{Harjulehto2019}
\end{remark}

We shall require the following lemma whose proof may be found in the statements of \cite[Lemma 3.7.6]{Harjulehto2019}, \cite[Lemma 4.1.7]{Harjulehto2019}, \cite[Lemma 4.2.4]{Harjulehto2019} and \cite[Proposition 2.4.9]{Harjulehto2019}.

\begin{lemma}
Let $E$ be an open set in $\RR^{N+1}$. Let $\varphi\in \Phi(E)$. If $\varphi$ satisfies \descref{A0}{A0}, \descref{A1}{A1}, \descref{A2}{A2}, \descref{aInc$_p$}{aInc$_p$} and \descref{aDec$_q$}{aDec$_q$}, then $\varphi^*$ satisfies \descref{A0}{A0}, \descref{A1}{A1}, \descref{A2}{A2}, and {\bfseries{{aInc}$_{q'}$}}.
\end{lemma}

The estimates for sparse operators will require the boundedness of the maximal function on the generalized Orlicz spaces. The proof is deferred to \cref{append2}.

\begin{lemma}\label{genOrcMaxibou}
Let $E$ be an open set in $\RR^{N+1}$ and let $\varphi\in\Phi(E)$. If $\varphi$ satisfies \descref{A0}{A0}, \descref{A1}{A1}, \descref{A2}{A2}, \descref{aInc$_p$}{aInc$_p$} and \descref{aDec$_q$}{aDec$_q$} with $1<p\leq q<\infty$, then we have 
\begin{align}
    \|\mathcal{M}f\|_{L^{\varphi(\cdot)}(E)}\leq C \|f\|_{L^{\varphi(\cdot)}(E)},\,f\in L^{\varphi(\cdot)}(E)
\end{align}
\begin{align}
    \|\mathcal{M}f\|_{L^{\varphi^*(\cdot)}(E)}\leq C \|f\|_{L^{\varphi^*(\cdot)}(E)},\,f\in L^{\varphi^*(\cdot)}(E).
\end{align}
\end{lemma}

\begin{remark}
As remarked upon earlier, the theory of generalized Orlicz spaces is an important synthesis with regards to the regularity theory of partial differential equations modelled on operators with Orlicz growth as they encompass a great variety of growth conditions such as 
\begin{itemize}
    \item $\varphi(z,s)=a(z)s^{p(z)}$.
    \item $\varphi(z,s)=s^{p(z)}\log(e+s)$.
    \item $\varphi(z,s)=s^p+a(z)s^{q}$.
    \item $\varphi(z,s)=s^p+a(z)s^p\log(e+s)$.
\end{itemize}
In particular, in the paper \cite{hasto_maximal_2021}, essentially optimal regularity results for these classes are recovered under optimal criteria on the variable exponent $p(z)$ and the phase switching factor $a(z)$ that matches the prescribed conditions \descref{A0}{A0}, \descref{A1}{A1}, \descref{A2}{A2}, \descref{aInc$_p$}{aInc$_p$} and \descref{aDec$_q$}{aDec$_q$}.
\end{remark}

\begin{remark}
It is worthwhile to note that regularity theory for parabolic equations in the case of Orlicz growth is not very well-developed. For example, the sharp parabolic analogues of \cite{ColMin2015,ColMinBounded2015} are not known. Hence, the maximal function estimates in parabolic settings requires a different approach that treats the time and space variables separately.
\end{remark}

\begin{lemma}
\label{sparse-estimates-genOrc}
Let $\GG$ be an $\eta$-sparse family. Let $\varphi\in\Phi(E)$. If $\varphi$ satisfies \descref{A0}{A0}, \descref{A1}{A1}, \descref{A2}{A2}, \descref{aInc$_p$}{aInc$_p$} and \descref{aDec$_q$}{aDec$_q$} with $1<p\leq q<\infty$, then for all $f\in L^{\varphi(\cdot)}(E)$, we have
\begin{enumerate}
    \item 
\begin{equation}
\label{estmates-for-sparse-1-genOrc}\|\mathscr{A}_{\GG}f\|_{L^{\varphi(\cdot)}(E)}\leq C_{N, p,q, \eta, \alpha,\beta}\|f\|_{L^{\varphi(\cdot)}(E)}.
\end{equation}
\item
Further, let $b\in L^{\infty}$ with small $(\delta, \mathfrak{K})$ norm where $\mathfrak{K}> r(\PP)$ for all $\PP\in \GG$. Then, we have
\begin{equation}
\label{estmates-for-sparse-2-genOrc}\|\mathscr{A}_{\GG}^{b}f\|_{L^{\varphi(\cdot)}(E)}+\|\mathscr{A}_{\GG}^{b, *}f\|_{L^{\varphi(\cdot)}(E)}\leq C_{N, p, q, \eta, \alpha,\beta} O(\delta)\|f\|_{L^{\varphi(\cdot)}(E)}.
\end{equation}
\end{enumerate}
\end{lemma}
\begin{proof}
Since the proof is very similar to the proof of Lemma~\ref{main-sparse-estimates}, we give a brief sketch of the proof.

\paragraph{\it\underline{Proof of \cref{estmates-for-sparse-1-genOrc}} :}
The proof relies on duality. Let $g\in L^{\varphi^*(\cdot)}(E)$ be a non-negative function. Then
\begin{align*}
\int_{\mathbb{R}^{N+1}} \mathscr{A}_{\GG}f(z) g(z)\, dz &\leq \sum_{\TT\in \GG} \left(\frac{1}{|\TT|}\int_{\TT}|f| \right) \left(\frac{1}{|\TT|}\int_{\TT} g \right) |\TT| \\
&\lesssim_{\eta} \sum_{\TT\in \GG} \left(\frac{1}{|\TT|}\int_{\TT}|f| \right) \left(\frac{1}{|\TT|}\int_{\TT} g \right) |F_\TT|\\
&\lesssim_{\eta} \sum_{\TT\in \GG} \inf_{z\in \PP} \mathcal{M}f(z)\inf_{z\in \PP} \mathcal{M}g(z) |F_\TT|\\
&\lesssim_{\eta} \sum_{\TT\in \GG}  \int_{F_{\TT}}\mathcal{M}f(z) \mathcal{M}g(z)\, dz\\
&\lesssim_{\eta} \int_{\mathbb{R}^{N+1}}\mathcal{M}f(z) \mathcal{M}g(z)\, dz\stackrel{\redlabel{b}{b}}{\lesssim_{\eta}} \|\mathcal{M}f\|_{L^{\varphi(\cdot)}(\mathbb{R}^{N+1})} \|\mathcal{M}g\|_{L^{\varphi^*(\cdot)}(\mathbb{R}^{N+1})}\\
&\stackrel{\redlabel{c}{c}}{\lesssim}_{\eta, \omega} \|f\|_{L^{\varphi(\cdot)}(\mathbb{R}^{N+1})} \|g\|_{L^{\varphi^*(\cdot)}(\mathbb{R}^{N+1})},
\end{align*} where \redref{b}{b} follows from \cref{holdergenOrc} and \redref{c}{c} follows from \cref{genOrcMaxibou}.
Now the proof is finished using duality (\cite[Theorem 3.4.6]{Harjulehto2019}).

\paragraph{\it\underline{Proof of \cref{estmates-for-sparse-2-genOrc}} :}
In \cref{main-sparse-estimates}, we have proved the following pointwise estimate
\begin{align*}
\mathscr{A}_{\GG}^{b, *}f(z)\leq c_{\eta, N}\, O(\delta) \mathscr{A}_{\GG}(\mathscr{A}_{\tilde{\GG}}|f|)(z),
\end{align*}
where $\tilde{\GG}$ is a $c_{\eta}$-sparse family containing $\GG$. Now \cref{estmates-for-sparse-1-genOrc} implies the following
\begin{align*}
\|\mathscr{A}_{\GG}^{b, *}f\|_{L^{\varphi(\cdot)}(E)}&\leq c_{\eta, N}\, O(\delta) \|\mathscr{A}_{\GG}(\mathscr{A}_{\tilde{\GG}}|f|)\|_{L^{\varphi(\cdot)}(E)}\\
&\leq c_{\eta, N, \omega}\, O(\delta) \|\mathscr{A}_{\tilde{\GG}}(f)\|_{L^{\varphi(\cdot)}(E)}\\
&\leq c_{\eta, N, \omega}\, O(\delta) \|f\|_{L^{\varphi(\cdot)}(E)}.  
\end{align*}
Thus the proof is complete.
\end{proof}

\begin{proposition}\label{estonball3}
 Let $\varphi$ satisfies \descref{A0}{A0}, \descref{A1}{A1}, \descref{A2}{A2}, \descref{aInc$_p$}{aInc$_p$} and \descref{aDec$_q$}{aDec$_q$} with $1<p\leq q<\infty$. There exists $C=C(\text{data})$ and $r_0=r_0(\text{data})$ such that if $u\in C_0^\infty(\RR^N\times(0,\infty))$, $\text{spt}\,u\subset B_r\subset \RR^N\times(0,\infty)$ with $0<r<r_0$, then, for $i,j=1,2,\ldots,s_0$
\begin{align*}
    \|u_{x_ix_j}\|_{L^{\varphi(\cdot)}(B_r)}\leq C_{A, B, N, p,q,\alpha,\beta}\, \|\mathcal{L}u\|_{L^{\varphi(\cdot)}(B_r)}.
\end{align*}
\end{proposition}

\begin{proof}
The proof is exactly similar to the proof of \cref{estonball1}. The only difference is we apply \cref{sparse-estimates-genOrc} to the pointwise representation in \cref{represent_pointwise_bound}.
\end{proof}

Now, the proof of \cref{main-theorem} is immediate.

\begin{proof}[Proof of Theorem~\ref{main-theorem}] In a first step, one obtains the following estimate by a standard argument using cut-off functions and an interpolation lemma, for example, see \cite[Theorem 9.11]{GT2001}.

\begin{align}\label{gradest6}
    \sup\limits_{i,j}\|u_{x_ix_j}\|_{L^{\varphi(\cdot)}(B_r)} &\leq C_{A, B, N, p, q, \alpha,\beta}\, \left(\|\mathcal{L}u\|_{L^{\varphi(\cdot)}(B_{2r})}+\frac{\|u\|_{L^{\varphi(\cdot)}(B_{2r})}}{r^2}\right).
\end{align}

Now, using the fact that
\begin{align*}
    Yu=\mathcal{L}u-\sum_{i,j=1}^{s_0}a_{ij}(z)u_{x_ix_j}
\end{align*}
we readily obtain through \cref{gradest6} that

\begin{align}\label{gradest7}
    \|Yu\|_{L^{\varphi(\cdot)}(B_r)} &\leq C_{A, B, N, p,q,\alpha,\beta}\, \left(\|\mathcal{L}u\|_{L^{\varphi(\cdot)}(B_{2r})}+\frac{\|u\|_{L^{\varphi(\cdot)}(B_{2r})}}{r^2}\right).
\end{align}

Now we deduce the estimate in \cref{main-theorem} by covering $\Om'$ with balls $B$ of radius $R/2$ where where $R$ has a further restriction $R<\text{dist}(\Om',\Om)$, other than the restriction from \cref{estonball3}

\end{proof}

\appendix

\section{Boundedness of maximal function on Generalized Orlicz spaces}\label{append2}

We sketch below the proof of boundedness of maximal functions in the case of $\varphi\in\Phi(\RR^{N+1})$. The case of $E\subset \RR^{N+1}$ is similar. We require the {\bf unit ball property} which asserts that

\begin{lemma}\label{unitballprop}
Let $\varphi\in\Phi(\RR^{N+1})$. Then
\begin{align*}
    \|f\|_{L^{\varphi(\cdot)}(\RR^{N+1})}<1 \implies \rho_{\varphi(\cdot)}(f)\leq 1 \implies \|f\|_{L^{\varphi(\cdot)}(\RR^{N+1})}\leq 1.
\end{align*}
\end{lemma}

\begin{lemma}\label{maxvarphibound}
Let $\varphi\in\Phi(\RR^{N+1})$. If $\varphi$ satisfies \descref{A0}{A0}, \descref{A1}{A1}, \descref{A2}{A2}, \descref{aInc$_p$}{aInc$_p$} and \descref{aDec$_q$}{aDec$_q$} with $1<p\leq q<\infty$, then we have 
\begin{align}
    \|\mathcal{M}f\|_{L^{\varphi(\cdot)}(\RR^{N+1})}\leq C \|f\|_{L^{\varphi(\cdot)}(\RR^{N+1})},\,f\in L^{\varphi(\cdot)}(\RR^{N+1})
\end{align}
\begin{align}
    \|\mathcal{M}f\|_{L^{\varphi^*(\cdot)}(\RR^{N+1})}\leq C \|f\|_{L^{\varphi^*(\cdot)}(\RR^{N+1})},\,f\in L^{\varphi^*(\cdot)}(\RR^{N+1}).
\end{align}
\end{lemma}

\begin{proof}

{{\it\underline{Step 1}} (Boundedness of maximal function on ultraparabolic $L^p$ spaces):} 

As described in the beginning of \cref{sparsedomprop}, there exist finitely many dyadic grids $\{\mathscr{U}^t\}_{t=1}^{L}$ such that 
\begin{align}
\mathcal{M}f\lesssim \sum_{t=1}^{L} \mathcal{M}_{\mathscr{U}^t}f,
\end{align} where $\mathcal{M}_{\mathscr{U}^t}$ is a dyadic maximal function. Now, the boundedness of maximal function follows from the boundedness of a dyadic maximal function whose proof may be found in \cite[Eq. 7.8]{HytKai}. Recall that we are strongly using the fact that the ultraparabolic cubes are comparable to the balls defined in \cref{balldef}.

{{\it\underline{Step 2}} (Key estimate):} 

The following lemma from \cite{Harjulehto2019} goes through as is for the homogeneous space $(\RR^{N+1},d,|\cdot|)$.

\begin{theorem}(\cite[Theorem 4.3.2]{Harjulehto2019})
Let $\varphi\in\Phi(\RR^{N+1})$. If $\varphi$ satisfies \descref{A0}{A0}, \descref{A1}{A1}, \descref{A2}{A2}, \descref{aInc$_p$}{aInc$_p$} with $1<p<\infty$, then there exists $\beta>0$ and $h\in L^1(\RR^{N+1})\cap L^\infty(\RR^{N+1})$ such that
\begin{align}\label{keyest}
    \varphi\left(x, \beta\fint_{B}|f|\,dy\right)^{\frac 1p}\leq \fint\varphi(y,|f|)^{\frac 1p}\,dy + h(x)^{\frac 1p} + \fint_B h(y)^{\frac 1p}\,dy,
\end{align} for every ball $B\subset\RR^{N+1}$, $x\in B$ and $f\in L^{\varphi(\cdot)}(\RR^{N+1})$ such that $\rho_{\varphi(\cdot)}(f)\leq 1$.
\end{theorem}

{{\it\underline{Step 3}}:} 
Taking supremum over all balls in the inequality \cref{keyest}, we obtain
\begin{align}\label{keyest2}
    \varphi\left(x, \beta\mathcal{M}f(x)\right)^{\frac 1p}\lesssim \mathcal{M}(\varphi(\cdot,f)^{\frac 1p})(x) + \mathcal{M}(h^{\frac 1p})(x),
\end{align}for every  $x\in \RR^{N+1}$ and $f\in L^{\varphi(\cdot)}(\RR^{N+1})$ such that $\rho_{\varphi(\cdot)}(f)\leq 1$. Here, we use the fact that $h(x)^{\frac 1p}\leq \mathcal{M}(h^{\frac 1p})(x)$.

{{\it\underline{Step 4}}:} 
Let $f\in L^{\varphi(\cdot)}(\RR^{N+1})$ such that $\|f\|_{L^{\varphi(\cdot)}(\RR^{N+1})}>0$. Choose $\varepsilon:=\frac{1}{2\|f\|_{L^{\varphi(\cdot)}(\RR^{N+1})}}$, then by \cref{unitballprop}, it holds that $\rho_{\varphi(\cdot)}(\ve f)\leq 1$ so that we may apply \cref{keyest2} to obtain
\begin{align}\label{keyest3}
    \varphi\left(x, \beta\ve\mathcal{M}f(x)\right)^{\frac 1p}\lesssim \mathcal{M}(\varphi(\cdot,\ve f)^{\frac 1p})(x) + \mathcal{M}(h^{\frac 1p})(x).
\end{align}
Raising both sides of \cref{keyest3} to power $p$ and integrating over $\RR^{N+1}$, we get
\begin{align}\label{keyest4}
    \int\limits_{\RR^{N+1}}\varphi\left(x, \beta\ve\mathcal{M}f(x)\right)\,dx & \lesssim \int\limits_{\RR^{N+1}}\left(\mathcal{M}(\varphi(\cdot,\ve f)^{\frac 1p})(x)\right)^p\,dx + \int\limits_{\RR^{N+1}}\left(\mathcal{M}(h^{\frac 1p})(x)\right)^p\,dx\nonumber\\
    &\leq \int\limits_{\RR^{N+1}}\varphi(\cdot,\ve f)(x)\,dx + \int\limits_{\RR^{N+1}}h(x)\,dx\nonumber\\
    &=\rho_{\varphi(\cdot)}(f) + \|h\|_{L^1(\RR^{N+1})}\nonumber\\
    &\leq 1 + \|h\|_{L^1(\RR^{N+1})} = C_1,
\end{align} where in the second inequality, we invoke boundedness of maximal function on $L^p(\RR^{N+1})$ from Step 1. The inequality \cref{keyest4} is the same as $\rho_{\varphi(\cdot)}(\beta\ve\mathcal{M}f)\leq C_1$ so that by {\bfseries{{aInc}$_{1}$}}(which is a consequence of convexity and $\phi(\cdot,0)=0$), we have $\rho_{\varphi(\cdot)}\left(\frac{\beta}{L C_1}\ve\mathcal{M}f\right)\leq 1$. Once again, applying \cref{unitballprop}, we get
$\|\frac{\beta}{L C_1}\ve\mathcal{M}f\|_{L^{\varphi(\cdot)}}\leq 1$ so that $\|\mathcal{M}f\|_{L^{\varphi(\cdot)}}\leq \frac{C_1 L}{\ve\beta}=\frac{2C_1 L}{\beta}\|f\|_{L^{\varphi(\cdot)}}$, which is the required result.

{{\it\underline{Step 5}}:} The result for $\varphi^*(\cdot)$ follows by the fact that it satisfies \descref{A0}{A0}, \descref{A1}{A1}, \descref{A2}{A2}, and {\bfseries{{aInc}$_{q'}$}}

\end{proof}

\medskip

\subsection*{Acknowledgement} We are thankful to Dr. Karthik Adimurthi and Dr. Agnid Banerjee for many fruitful discussions.

\providecommand{\bysame}{\leavevmode\hbox to3em{\hrulefill}\thinspace}
\providecommand{\MR}{\relax\ifhmode\unskip\space\fi MR }
\providecommand{\MRhref}[2]{%
  \href{http://www.ams.org/mathscinet-getitem?mr=#1}{#2}
}
\providecommand{\href}[2]{#2}

\end{document}